\documentclass{amsart}
\usepackage{mathtools}

\usepackage[all]{xy}
\usepackage{amssymb}
\usepackage{amsmath}
\usepackage{cite}
\usepackage{fullpage}

\usepackage{url}
\usepackage{hyperref}
\numberwithin{equation}{subsection}
\numberwithin{figure}{subsection}
\usepackage{longtable}
\usepackage{tikz}
\usepackage{tikz-cd}
\usepackage{times}

\usepackage{setspace}
\newtheorem{theorem}{Theorem}[section]
\newtheorem{lemma}[theorem]{Lemma}
\newtheorem{proposition}[theorem]{Proposition}
\newtheorem{corollary}[theorem]{Corollary}
\theoremstyle{definition}
\newtheorem{definition}[theorem]{Definition}
\newtheorem*{remark*}{Remark}
\newtheorem{remark}[theorem]{Remark}
\newtheorem*{remarks*}{Remarks}
\newtheorem*{definition*}{Definition}
\newtheorem{example}{Example}[section]
\newtheorem{question*}{Question}[section]

\renewcommand{\tilde}{\widetilde}

\newcommand{\CC}{\mathbb{C}}
\newcommand{\QQ}{\mathbb{Q}}

\newcommand{\ZZ}{\mathbb{Z}}

\newcommand{\FF}{\mathbb{F}}
\newcommand{\PP}{\mathbb{P}}

\newcommand{\ga}{G^{\text{arith}}}
\newcommand{\gm}{G^{\text{geom}}}

\newcommand\restr[1]{\raisebox{-.5ex}{$|$}_{#1}}

\DeclareMathOperator{\Gal}{Gal}
\DeclareMathOperator{\sgn}{sgn}

\DeclareMathOperator{\Frob}{Frob}
\DeclareMathOperator{\im}{im}
\DeclareMathOperator{\id}{id}

\DeclareMathOperator{\Aut}{Aut}

\newcommand{\customlangle}{%
\begin{tikzpicture}[baseline=0.5ex]
    \draw coordinate (a) at (.35ex,2.2ex);
    \draw coordinate (b) at (0,1.1ex);
    \draw coordinate (c) at (.35ex,0);
    \draw[line width=0.12ex,line cap=round] (a) -- (b) -- (c);
\end{tikzpicture}%
}
\newcommand{\customrangle}{%
\begin{tikzpicture}[baseline=0.5ex]
    \draw coordinate (a) at (0,2.2ex);
    \draw coordinate (b) at (.35ex,1.1ex);
    \draw coordinate (c) at (0,0);
    \draw[line width=0.12ex,line cap=round] (a) -- (b) -- (c);
\end{tikzpicture}%
}

\begin{document}
	
\title{Galois theory of Quadratic Rational functions with periodic critical points}

\author{\"{O}zlem Ejder}
\address{Ko\c{c} University, College of Science}
%\curraddr{Department of Mathematics and Statistics,
%Case Western Reserve University, Cleveland, Ohio 43403}
\email{ozejder@ku.edu.tr}
\thanks{The author was supported by BAGEP Award of the Science Academy in Turkey.}


\date{}


\begin{abstract}
Given a number field $k$, and a quadratic rational function $f(x) \in k(x)$, the associated arboreal representation of the absolute Galois group of $k$ is a subgroup of the automorphism group of a regular rooted binary tree. Boston and Jones conjectured that the image of such a representation for $f \in \ZZ[x]$ contains a dense set of settled elements. An automorphism is settled if the number of its orbits on the $n\text{th}$ level of the tree remains small as $n$ goes to infinity.

In this article, we exhibit many quadratic rational functions whose associated Arboreal Galois groups are not densely settled. These examples arise from quadratic rational functions whose critical points lie in a single periodic orbit. To prove our results, we present a detailed study of the iterated monodromy groups (IMG) of $f$, which also allows us to provide a negative answer to Jones and Levy's question regarding settled pairs.

Furthermore, we study the iterated extension $k(f^{-\infty}(t))$ generated by adjoining to $k(t)$ all roots of $f^n(x) = t$ for $n \geq 1$
for a parameter $t$. We call the intersection of $k(f^{-\infty}(t))$ with $\bar{k}$, the field of constants associated with $f$. When one of the two critical points of $f$ is the image of the other, we show that the field of constants is contained in the cyclotomic extension of $k$ generated by all $2$-power roots of unity. In particular, we prove the conjecture of Ejder, Kara, and Ozman regarding the rational function $\frac{1}{(x-1)^2}$. 

\end{abstract}

\date{}
\maketitle
\section{introduction}
Given a polynomial $f(x) \in k[x]$ over a field $k$, we call $f$ settled if the number of irreducible factors of its $n$th iterate remains small as $n$ grows. Boston and Jones \cite{BJ-pams} \cite{BJ-dedicata} \cite{BJ-quarterly} conjectured that a quadratic polynomial over a finite field is settled. This conjecture has implications for the associated arboreal Galois groups of $f$. 

Let $f(x)$ be a rational function of degree two in $k(x)$, where $k$ is a field. Let $\bar{k}$ be a separable closure of $k$. For all $n\geq 1$, $f^n$ denotes the $n\text{th}$ iterate of $f$. Let $a \in k$ and assume that  $f^n(x)-a$ has $2^n$ distinct roots for all $n\geq 1$. We form a regular rooted binary tree $T$ from the pre-image set $f^{-n}(a)$, where the leaves are the elements of the pre-image set and two leaves $v$ and $w$ are directly connected by an edge if $f(v)=w$. Taking the limit as $n \to \infty$, we obtain an infinite tree $T$ on which the absolute Galois group $G_k:=\Gal(\bar{k}/k)$ of $k$ acts naturally. This yields a representation $\rho_{f,a} \colon G_k \to \Omega$, where $\Omega$ denotes the automorphism group of $T$.  We call the image of $\rho_{f,a}$, the arboreal Galois group attached to $(f,a)$ and denote it by $G_{f,a}$. Arboreal representations have applications to certain prime density questions, and they are widely studied (see \cite{Stoll92},\cite{Odoni2}, \cite{Odoni85}, \cite{Jonessurvey}, \cite{JM14},\cite{S16}). For explicit computations of these Galois groups, see for instance \cite{BEK},\cite{Arborealcubic} ,\cite{EjderBelyi}, \cite{EKO}, \cite{arithmeticbasilica}, \cite{BHWL17}, \cite{JKLMTW}. 

By the conjecture of Boston and Jones, the Frobenius elements in the associated arboreal Galois groups of a quadratic rational function in $\ZZ[x]$ exhibit a specific cycle pattern. More precisely, the Frobenius elements in the arboreal Galois groups of $f$ are settled. We call an automorphism of $T$ \emph{settled} if the number of its orbits on $n$th level remains small as $n \to \infty$. See ~\ref{defn:settled} for the exact definition. By the Chebotarev density theorem, the set of Frobenius elements is dense in the associated arboreal Galois groups of $f$. Therefore, the conjecture of Boston and Jones (\cite{BJ-dedicata}, \cite{BJ-quarterly}, \cite{BJ-pams}) implies that the set of settled elements is dense in $G_{f,a}$. A subgroup of $\Omega$ containing a dense set of settled elements is called densely settled. Recently, Cortez and Lukina \cite{CortezLukina} introduced a geometric group theoretic approach to study the property of being densely settled for quadratic polynomials. In this article, we address the question of whether the arboreal Galois groups are densely settled for a non-polynomial rational function. This question coincides with the question posed by Jones and Levy \cite{JonesLevy} on settled pairs. More precisely, an example of an arboreal Galois group which is not densely settled provides a negative answer to their question. We show that there exist many quadratic rational functions $f$ for which the groups $G_{f,a}$ do not contain any settled elements.

Let $k$ be a field and let $f(x) \in k(x)$ be a rational function. The post-critical set of $f$ is defined as
$$P:=\{f^n(c)\mid c \text{ is a critical point of } f, \text{ } n\geq 1\}.$$
If the set $P$ is finite, the map $f$ is called post-critically finite (PCF). In this article, we study the PCF quadratic rational maps whose critical points all lie in a single periodic orbit. An explicit example of such a function is $f(x)=\frac{1}{(x-1)^2}$, whose post-critical set has $3$ points: $\{1,f(1),f^2(1)\}$.

\begin{theorem} \label{thm:intro1}
Let $k$ be a number field. Let $f(x) \in k(x)$ be a quadratic rational function whose post-critical set is $\{p_1, p_2, \ldots, p_r\}$ where $f(p_{k-1})=p_k$ for all $k \in \{1, \ldots, r\}$. Here the indices are considered modulo $r$. Let $p_r$ and $p_{s-1}$ be the critical points of $f$.  If one of the following conditions holds:
\begin{enumerate}
\item $r$ is odd,
\item $r$ is even, $s$ is odd and $\frac{r}{\gcd(r,s-1)}$ is odd,
\end{enumerate}
then the arboreal representation of $(f,a)$ does not contain any settled elements for any $a \in \PP^1(k) \setminus \{p_1,p_2, \ldots,p_r\}$.
\end{theorem}
Let $f(x)=\frac{1}{(x-1)^2}$. By Theorem~\ref{thm:intro1}, the arboreal representation of $(f,a)$ for $a \in k \setminus \{0,1\}$ does not contain any settled elements.

The proof of Theorem~\ref{thm:intro1} requires a detailed study of the geometric and arithmetic iterated monodromy groups of $f$.
For each $n\geq 1$, one can view the iterate $f^{n}$ of $f$ as a morphism from the projective line $\PP^1_k$ to itself. Excluding the post-critical set $$P:=\{f^n(c) \mid c \text{ is a critical point of } f, \text{ } n\geq 1\}, $$ one obtains an unramified covering of $\PP^1_k \setminus P$ and the monodromy group of this cover is a quotient of the {\'e}tale fundamental group of $\PP_k^1\setminus P$. This process yields two monodromy groups depending on the field we are working on: $k$ (arithmetic) or the separable closure $\bar{k}$ (geometric) of $K$. Furthermore, by taking the limit as $n \to \infty$, we obtain the geometric ($\gm(f)$) and arithmetic iterated monodromy groups ($\ga(f)$) associated with $f$. In essence, the arboreal representations can be viewed as specializations of the arithmetic iterated monodromy groups. We use IMG as a shorthand for iterated monodromy group, and we note that IMGs are always profinite unless stated otherwise.

 Pink \cite{Pinkpolyn}, \cite{Pinkrational} provided a comprehensive analysis of monodromy groups in the cases where $f$ is a quadratic polynomial and $f$ is a quadratic rational function with an infinite post-critical set. The case involving PCF rational functions remains unsolved. In collaboration \cite{EKO} with Kara and Ozman, we studied the IMGs associated with the function $f(x)=\frac{1}{(x-1)^2}$. The current article extends this study to the PCF rational functions of degree $2$ whose critical points lie in a single periodic orbit. 

Our key observation in proving Theorem~\ref{thm:intro1} is Proposition~\ref{sodo}, where we prove that if a self-similar group contains a settled element, then it must contain an odometer. An automorphism of a regular rooted binary tree is called an odometer if it has only one orbit on each level of the tree. Since both the geometric and arithmetic IMGs are self-similar, it becomes natural to study the odometers in these groups. 

\begin{theorem}
Let $k$ be a number field. Let $f(x) \in k(x)$ be a quadratic rational function whose post-critical set is $\{p_1, p_2, \ldots, p_r\}$, where $f(p_{k-1}) = p_k$ for all $k \in \{1, \ldots, r\}$, with indices modulo $r$. Let $p_r$ and $p_{s-1}$ be the critical points of $f$. Then the geometric IMG of $f(x)$ contains odometers if and only if either:
\begin{itemize}
\item both $r$ and $s$ are even, or
\item $s$ is odd and $\frac{r}{\gcd(r, s-1)}$ is even.
\end{itemize}
\end{theorem}
For a parameter $t$, consider the iterated extension obtained by the union over $n$ of $k(f^{-n}(t))$, the field extension generated by the roots of $f^n(x)-t=0$ over $k(t)$. The intersection of this extension with $\bar{k}$ is called the field of constants associated with $f$. Pink \cite{Pinkpolyn} showed that in the case of a PCF quadratic polynomial $f$, the field of constants of $f$ is a subfield of the cyclotomic field generated by the $2^n$th roots of unity for all $n\geq 1$. Hamblen and Jones \cite[Corollary~2.4]{HamblenJones} showed that for quadratic rational maps whose critical points have periodic orbits, a quadratic extension of the union over $n$ of $k(f^{-n}(b))$ for all $b \in k$, except for finitely many exceptions (obtained by adjoining the critical points) contains the $2^n$th roots of unity for all $n \geq 1$. 

We prove the following result:

\begin{theorem}\label{thm:intro2}
Let $k$ be a number field. Let $f(x) \in k(x)$ be a quadratic rational function whose post-critical set is $\{p_1,p_2, \ldots,p_r\}$ where $f(p_{m-1})=p_m$ for $1\leq m \leq r$. Here the indices are considered modulo $r$. Let $p_r$ and $p_1$ be the critical points of $f$. Then the field of constants associated with $f$ is contained in the cyclotomic field $k(\zeta_{2^n} : n\geq 1)$. 
\end{theorem}

Theorem~\ref{thm:intro2} proves \cite[Theorem~1.2 ]{EKO} unconditionally. 
\begin{corollary}
Let $k$ be a number field and let $f(x)=\frac{1}{(x-1)^2} \in k(x)$. The field of constants associated with $f$ is the cyclotomic field $k(\zeta_{2^n} \colon n\geq 1)$.
\end{corollary}

We outline the article as follows. In Section~\ref{gIMG}, we determine the geometric IMG up to conjugacy. In Section~\ref{sec:dsettled}, we study the odometers and settled elements in $\gm$ and $\ga$, leading to the proof of Theorem~\ref{thm:intro1}. Section~\ref{sec:structure} presents a further study of the group theoretical properties of $\gm$. Finally, in 
Section~\ref{maq}, we prove Theorem~\ref{thm:intro2} using the results from Section~\ref{sec:structure}.

\section{Acknowledgements}
We thank Rafe Jones and Olga Lukina for their discussions on Proposition 5.6. The author was supported by the Science Academy’s Young Scientist Awards Program (BAGEP) in Turkey.

\section{Notation and Background}\label{background}
\subsection{Automorphism group of a regular rooted binary tree}\ \\
Let $X=\{0,1\}$ and consider the regular rooted binary tree $T$ whose vertices are the finite words over $X$. The root of $T$ is the empty word, and a word $v$ is connected to all words of the form $vx$ for $x \in X$. For any integer $n \geq 1$, let $T_n$ denote the finite rooted subtree whose vertices are the words of length at most $n$. The set of words of length $n$ is called the \emph{level} $n$ of $T$, and we denote by $V_n$ the set of vertices at that level. 

An isomorphism between two regular rooted binary trees $T$ and $T'$ is a bijection of their vertices that preserves the tree structure. When $T=T'$, the group of such isomorphisms is called the automorphism group of $T$, and it is denoted by $\Omega$. Similarly, we let $\Omega_n$ denote the group of automorphisms of the finite tree $T_n$. For every $n\geq 1$, we write $\pi_n$ for the natural projection
\begin{equation}\label{not:pi}
\pi_{n}\colon \Omega\to \Omega_n,
\end{equation}
which corresponds to restricting the action of an element of
$\Omega$ to the subtree $T_n$ consisting of the levels $0, 1, \ldots, n$. Similarly, for any $m \geq n$, we denote the natural projection $\Omega_m \to \Omega_n$ by $\pi_{m,n}$. For clarity, we will abuse notation and write $\pi_n$ whenever the domain is clear.

For $n \geq 1$, we denote the image of an element $\gamma \in \Omega$ under $\pi_n$ by $\gamma \restr{T_n}$. Note that two automorphisms $\gamma$ and $\gamma'$ in $\Omega$ are equal if and only if $\gamma \restr{T_n} = \gamma' \restr{T_n}$ for all $n \geq 1$. Let $H$ be a subgroup of $\Omega$. For each $n \geq 1$, we define $H_n := \pi_n(H) \subset \Omega_n$.

The set
\[
\delta T := \left\{ (v_n)_{n \geq 1} \in \prod_{n \geq 1} V_n \,\middle|\, v_{n+1} = v_n x \text{ for some letter } x \in X, \text{ for all } n \geq 1 \right\}
\]
is called the \emph{boundary} of $T$, and each element of $\delta T$ is called a \emph{path}. Let $v = (v_n)_{n \geq 1}$ and $w = (w_n)_{n \geq 1}$ be two paths in $\delta T$. We define a metric on $\delta T$ as follows:

If $v = w$, then set $d(v, w) := 0$. If $v \neq w$, then define
\[
d(v, w) := \frac{1}{2^n},
\]
where $n$ is the largest nonnegative integer such that $v_n = w_n$ in $V_n$.

Using the metric $d$ on $\delta(T)$, we can define a metric on $\Omega$ as follows: for any $\gamma, \gamma' \in \Omega$,  
\begin{equation}\label{def:dd}
d(\gamma, \gamma') := \sup_{v \in \delta T} d\big( (v)\gamma, (v)\gamma'\big).
\end{equation}
In other words, for $\gamma \neq \gamma'$, we have $d(\gamma, \gamma') = \frac{1}{2^n}$, where $n$ is the highest level of $T$ on which $\gamma$ and $\gamma'$ agree.

For every $n \geq 1$, the set
\[
U_n := \left\{ \gamma \in \Omega \;\middle|\; d(\gamma, \id) \leq \frac{1}{2^n} \right\}
\]
is open in the uniform topology defined by $d$. Moreover, the collection $\{U_n\}_{n \geq 1}$ forms a fundamental system of open neighborhoods of the identity element in $\Omega$. The set $U_n$ coincides with the kernel of the homomorphism $\pi_n$, and $\Omega$ is the inverse limit of the system $((\Omega_n)_{n \geq 1}, (\pi_{m,n})_{m \geq n})$. Hence, the profinite topology on $\Omega$ agrees with the uniform topology. In particular, $\Omega$ is complete under the profinite (or uniform) topology.

\subsection{Closed subgroups of $\Omega$}\ \\

Throughout this article, $\ZZ_2$ denotes the ring of 2-adic integers, and $\ZZ_2^\times$ denotes its group of units. We represent an element $k \in \ZZ_2$ by a sequence $(k_n)_{n \geq 1}$, where each $k_n \in \ZZ$ and satisfies the congruence $k_n \equiv k_{n-1} \pmod{2^n}$ for all $n \geq 1$.

Let $\gamma \in \Omega$ and let $k = (k_n)_{n \geq 1} \in \ZZ_2$. For any $m \geq 1$, the difference $(k_{m+1} - k_m)$ is divisible by $2^m$. Since $\Omega_m$ has exponent $2^m$, the automorphism $\gamma^{k_{m+1} - k_m}$ lies in the kernel of the reduction map $\pi_m\colon \Omega \to \Omega_m$. This implies that the sequence $(\gamma^{k_n})_{n \geq 1}$ is Cauchy in $\Omega$. As $\Omega$ is complete under the profinite topology, the sequence $(\gamma^{k_n})$ converges in $\Omega$. We denote its limit by $\gamma^k$.

The closure of the cyclic group generated by an element $\gamma \in \Omega$ is defined as
\[
\customlangle\customlangle \gamma \customrangle\customrangle := \{ \gamma^k \mid  k \in \ZZ_2 \},
\]
and the closure of the $\gamma$-orbit of any $v \in \delta T$ is given by the orbit of $v$ under the closed subgroup
$\customlangle\customlangle \gamma \customrangle\customrangle$, i.e.,
\[
\{(v) \gamma^k  \mid k \in \ZZ_2 \}.
\]
For $c_1, \ldots, c_k \in \Omega$, let $\customlangle\customlangle c_1, \ldots, c_k \customrangle\customrangle$ denote the topological closure of the subgroup $\langle c_1, \ldots, c_k \rangle$ in $\Omega$. We say that a subgroup $H$ is topologically generated by $c_1, \ldots, c_k$ if
\[
H = \customlangle\customlangle c_1, \ldots, c_k \customrangle\customrangle.
\]
Note that if $H$ is a closed subgroup of $\Omega$, then $H$ is the inverse limit of the system $(H_n)_{n \geq 1}$ as $n \to \infty$.

\subsection{The Semi-direct Product Structure of $\Omega$}\ \\

Let $T$ be a regular rooted binary tree, and let $T_0$ and $T_1$ denote the subtrees rooted at the first level of $T$. We embed the direct product of the automorphism groups of $T_0$ and $T_1$ into $\Omega$ by sending $(\gamma_0, \gamma_1)$ to the automorphism of $T$ that acts trivially on the first level and acts as $\gamma_i$ on the subtree $T_i$ for $i = 0, 1$.

A vertex of $T_0$ at level $n$ is a vertex of $T$ written as $0v$ for some word $v$ of length $n$. Using this notation, we define an isomorphism from $T$ to $T_0$ by sending a vertex $w$ to $0w$. Similarly, an isomorphism from $T$ to $T_1$ is given by $w \mapsto 1w$. These induce isomorphisms on their respective automorphism groups. Via these identifications, we obtain the embedding
\[
\Omega \times \Omega \hookrightarrow \Omega.
\]
The image of this embedding is the set of automorphisms that act trivially on the first level of the tree. 

Therefore, $\Omega$ admits the semidirect product structure
\[
\Omega \simeq (\Omega \times \Omega) \rtimes \customlangle \sigma \customrangle,
\]
where $\sigma$ is the automorphism in $\Omega$ that permutes the two subtrees $T_0$ and $T_1$ rooted at level one. Hence, any element $\gamma \in \Omega$ can be written as $\gamma = (\gamma_0, \gamma_1)\tau$, where $\gamma_0, \gamma_1 \in \Omega$ and $\tau \in S_2$, the symmetric group on two letters.

The semidirect product is defined by the rule
\begin{equation}\label{semi}
(\gamma_0, \gamma_1)\tau \cdot (\gamma_0', \gamma_1')\tau' = \big(\gamma_0 \gamma'_{(0)\tau}, \gamma_1 \gamma'_{(1)\tau}\big)\tau\tau'.
\end{equation}
The inverse of an automorphism $(\gamma_0, \gamma_1)\tau$ is given by
\[
((\gamma_0, \gamma_1)\tau)^{-1} = \big(\gamma^{-1}_{(0)\tau}, \gamma^{-1}_{(1)\tau}\big)\tau^{-1}.
\]

Similarly, for any $n \geq 2$, we have
\[
\Omega_n \simeq (\Omega_{n-1} \times \Omega_{n-1}) \rtimes S_2,
\]
where $S_2 \simeq \customlangle \sigma \customrangle$.

The action of $\gamma = (\gamma_0, \gamma_1)\tau$ on $T$ is given as follows:
\[
(xv)\gamma = (x)\tau \cdot (v)\gamma_x
\]
for any $x \in X$ and any word $v$. We assume that a permutation in $S_n$ acts from right. With the notation above, we note that $\sigma = (\id, \id)\sigma \in \Omega$, where $\id$ denotes the identity automorphism in $\Omega$.
\begin{example}
Let $\gamma \in \Omega$ be defined recursively by $\gamma = (\gamma, \id)\sigma$. Then 
$(10)\gamma=00, (00)\gamma=11, (01)\gamma=10,$ and $(11)\gamma=01$.
\end{example}

\subsection{Signs}\ \\
For any $n \geq 1$, the group $\Omega_n$ acts faithfully on the $n$th level of the tree $T$, which allows us to embed $\Omega_n$ into the symmetric group $S_{2^n}$. We let $\sgn_n$ denote the sign of the induced permutation on level $n$, which defines a continuous homomorphism
\begin{equation}
\sgn_n\colon \Omega \to \{\pm 1\}.
\end{equation}

For each $n \geq 1$, the automorphism $\sigma \restr{T_n}$ is a product of $2^{n-1}$ disjoint $2$-cycles. Therefore, the sign of $\sigma$ at level $n$ is given by
\begin{equation}\label{sgn:sigma}
\sgn_n(\sigma) = 
\begin{cases}
-1 & \text{if } n = 1, \\
\;\;\,1 & \text{if } n > 1.
\end{cases}
\end{equation}

 \subsection{Odometers}\ \\
An automorphism $\alpha \in \Omega$ is called an \emph{odometer} if $\alpha$ acts as a $2^n$ cycle on every level $n$ of the tree. Equivalently, $\alpha$ is an odometer if $\alpha \restr{T_n}$ has order $2^n$ for all $n\geq 1$. The recursively defined element $\beta=(\beta,\id)\sigma$ is an example of an odometer and it is called the standard odometer.

If $\alpha$ is an odometer in $\Omega$, then $\alpha\restr{T_n}$ is a $2^n$-cycle, hence it is an odd permutation for all $n\geq 1$. Therefore, $\sgn_n(\alpha)=-1$ for all $n\geq 1$. Conversely, let $\alpha \in \Omega$ such that $\sgn_n(\alpha)=-1$ for all $n\geq 1$. If $\alpha\restr{T_{n-1}}$ is a $2^{n-1}$ cycle, then $\alpha\restr{T_n}$ is either a $2^{n}$ cycle or it is the product of two $2^{n-1}$-cycles. Since $\sgn_n(\gamma)=-1$,  $\gamma \restr{T_n}$ cannot be a product of two $2^{n-1}$ cycles. Hence we have the following proposition.
\begin{proposition}\cite[Proposition~1.6.2]{Pinkpolyn}\label{pink:odometer}
An element $\gamma \in \Omega$ is an odometer if and only if $\sgn_n(\gamma)=-1$ for all $n\geq 1$.
\end{proposition}
We note here that Proposition~\ref{pink:odometer} does not hold for $d$-ary trees where $d>2$. We also note that any conjugate of an odometer in $\Omega$ is again an odometer since the cycle structure is preserved under conjugacy.

\subsection{Self-similarity} \ \\
Let $V$ be the set of all vertices of $T$, i.e., \( V := \bigcup_n V_n \), and let $\gamma = (\gamma_0, \gamma_1)\tau$ be an automorphism of $T$. For $x \in X$, the \emph{section} of $\gamma$ at $x$ is uniquely defined by the relation
\begin{equation}\label{sec}
(xv)\gamma = (x)\tau(v)\gamma_x
\end{equation}
for any letter $x \in X$ and any word $v$ over $X$ (i.e., $v \in V$).  

Let $\gamma, \gamma' \in \Omega$. Using equations~\eqref{semi} and \eqref{sec}, we find that the section of $\gamma\gamma'$ at $x \in X$ is given by $\gamma_x \gamma'_{\gamma(x)}$, i.e.,
\[
(\gamma \gamma')_x = \gamma_x \gamma'_{(x)\gamma}.
\]

The sections of an automorphism $\gamma$ of $T$ are again automorphisms of $T$, under the identification of the subtree rooted at $x$ with the full tree $T$. Consequently, the sections of these sections are also automorphisms of $T$, and we recursively define the section of $\gamma$ at any vertex by
\begin{equation}
\gamma_{ux} := (\gamma_u)_x
\end{equation}
for any word $u$ and any letter $x \in X$. Hence, for any words $u$ and $v$,
\begin{equation}\label{section}
(uv)\gamma = (u)\gamma(v)\gamma_u.
\end{equation}

\begin{example}\label{ex:odo}
Let $\gamma \in \Omega$ be defined recursively by $\gamma = (\gamma, \id)\sigma$. Then the section of $\gamma$ at a vertex $v$ is the identity, except when $v$ is a word consisting only of $0$'s, in which case $\gamma_v = \gamma$.
\end{example}

A subgroup $H$ of $\Omega$ is called \emph{self-similar} if, for every automorphism $\gamma \in H$, its sections $\gamma_0$ and $\gamma_1$ also lie in $H$. In other words, $H$ is self-similar if its image under the isomorphism $\Omega \to (\Omega \times \Omega) \rtimes S_2$ is contained in $(H \times H) \rtimes S_2$.

We refer the reader to \cite{Nselfsimilarbook} and \cite{RZselfsimilarandbranchgt} for further details on self-similarity.

\subsection{Monodromy Groups}\label{mon}\ \\
Let $k$ be a number field and $\bar{k}$ be a separable closure of $k$.  Let $f\colon\PP^1_k \to \PP^1_k$ be a morphism of degree $d$ defined over $k$. We denote the set of critical points of $f$ by $C$ and the forward orbit of the points in $C$ by $P$, i.e.
\[ P:=\{ f^n(c) \mid n\geq 1, c \in C\}. \]
We call $P$ the \emph{post-critical set} of $f$. 
A rational function is called post-critically finite (PCF) if the post-critical set $P$ of $f$ is finite.
For $n\geq 1$, the iterate $f^n$ is a connected unramified covering of $\PP^1_k \backslash P$, hence it is determined by the monodromy action of $\pi_1^{\acute{e}t}(\PP_k^1 \backslash P, x_0)$ on $f^{-n}(x_0)$ up to isomorphism, where $x_0 \in \PP^1(k)\backslash P$. Let $T_{x_0}$ be the tree defined as follows: it is rooted at $x_0$, the leaves of $T_{x_0}$ are the points of $f^{-n}(x_0)$ for all $n\geq 1$, and the two leaves $p,q$ are connected if $f(p)=q$.  By taking the inverse limit over $n$, associated monodromy defines a representation 
\begin{equation}\label{eq:rho}
 \rho_{f,x_0}\colon \pi_1^{\acute{e}t}(\PP_k^1 \backslash P, x_0) \to \Omega(T_{x_0})
\end{equation} 
where $\Omega(T_{x_0})$ denotes the automorphism group of $T_{x_0}$. We note that if $y_0$ is another point in $\PP^1(k)\backslash P$, there is an isomorphism between the trees $T_{x_0}$ and $T_{y_0}$. Furthermore, $T_{x_0}$ and $T_{y_0}$ are isomorphic to the tree $T$ we defined earlier. Under these isomorphisms, $\im(\rho_{f, x_0})$ and $\im(\rho_{f, y_0})$ are conjugate subgroups of $\Omega$. Hence, we have:
\begin{equation}\label{eq:rho}
 \rho_{f}\colon \pi_1^{\acute{e}t}(\PP_k^1 \backslash P, x_0) \to \Omega
\end{equation} 

We call the image of $\rho_f$ the \emph{arithmetic iterated monodromy group} of $f$ and denote it by $G^{\text{arith}}(f)$. One can also study this representation over $\bar{k}$ and obtain 
\[ \pi_1^{\acute{e}t}(\PP_{\bar{k}}^1 \backslash P, x_0) \to \Omega(T).
\]
 We call the image of the map in this case the \emph{geometric iterated  monodromy group} and denote it by $\gm(f)$. The groups $\ga(f)$ and $\gm(f)$ fit into an exact sequence as follows:
 \begin{equation}\label{eq:exact}
   \begin{tikzcd}
1\arrow{r}  &  \pi_1^{\acute{e}t}(\PP_{\bar{k}}^1 \backslash P, x_0) \arrow{r} \arrow{d}& \pi_1^{\acute{e}t}(\PP_k^1 \backslash P, x_0) \arrow{r} \arrow{d} &\Gal(\bar{k}/ k) \arrow{r} \arrow{d} &1  \\
   1 \arrow{r} &   G^{\text{geom}}(f) \arrow{r} &G^{\text{arith}}(f) \arrow{r}  &\Gal(F /k) \arrow{r} &1 \\
     \end{tikzcd}
  \end{equation}
for some field extension $F$ of $k$. We call this field $F$ the field of constants for $f$. 

The groups $G^{\text{geom}} (f)$ and $G^{\text{arith}}(f)$ are profinite self-similar subgroups of $\Omega$. Also note that $G^{\text{geom}}(f)$ is a normal subgroup of $G^{\text{arith}}(f)$. 

\subsection{Arboreal Galois Groups} \ \\
Let $k$ be a number field, and let $\bar{k}$ be a separable closure of $k$. Let $f \colon \PP^1_k \to \PP^1_k$ be a morphism of degree $d$ defined over $k$. Fix $\alpha \in k$, and assume that the equation $f^n(x) = \alpha$ has $2^n$ distinct solutions in $\PP^1(\bar{k})$ for all $n \geq 1$. Then the set of preimages of $\alpha$ under the iterates of $f$ forms a regular rooted binary tree $T_\alpha$, and the absolute Galois group $G_k = \Gal(\bar{k}/k)$ acts naturally on this tree.

The tree $T_\alpha$ is isomorphic to the standard binary tree $T$, and under this isomorphism we obtain a representation
\[
\rho_{f,\alpha} \colon G_k \to \Omega.
\]
The image of $\rho_{f,\alpha}$ is called the \emph{arboreal Galois group} attached to the pair $(f, \alpha)$. Moreover, this group can be embedded into $\ga(f)$ under the specialization at $t = \alpha$. We will denote this group by $G_{f,\alpha}$.

\section{Determining the geometric iterated monodromy groups}\label{gIMG}
Let $k$ be a number field, and let $f$ be a rational function in $k(x)$ of degree $2$. We assume that the critical points of $f$ are periodic and lie in a single orbit under the iterated application of $f$.

Since $f$ has degree $2$, it has two critical points. We further assume that the post-critical set of $f$ has $r \geq 3$ elements. Let
\[
P = \{p_1, p_2, \ldots, p_r\}, \quad \text{where } p_i = f^i(p_r) \text{ for } i = 1, \ldots, r.
\]
This post-critical orbit can be illustrated as:
\begin{equation}\label{pcorbit}
p_r \to p_1 \to \ldots \to p_{s-1} \to \ldots \to p_s \to \ldots \to p_{r-1} \to p_r,
\end{equation}
for some $1 < s \leq r$ and $r \geq 3$, where the branch points of $f$ are $p_1$ and $p_s$.

When $f$ is defined over $\CC$, the geometric iterated monodromy group (IMG) of a PCF rational function $f$ is topologically finitely generated by the inertia generators $b_p$ for each $p \in P$, since it is the profinite closure of the discrete IMG of $f$. By Grothendieck’s theory, this result extends to any algebraically closed field of characteristic zero.

Therefore, if $f(x) \in k(x)$ is a PCF function defined over a number field $k$, then $\gm(f)$ is topologically finitely generated, and its generators satisfy certain recursive conjugacy relations.

We use \cite[Proposition~1.7.15]{Pinkpolyn} to describe a set of generators for $G^{\text{geom}}(f)$. For each $p \in P$, let $b_p$ denote the inertia generator at $p$ in the geometric IMG of $f$.

\begin{proposition}[\cite{Pinkpolyn}, Proposition~1.7.15]\label{Pink:gen}
Let $k$ be a number field, and let $f \in k(x)$ be a quadratic rational function with post-critical set $P$ as described in \eqref{pcorbit}. For any $1 \leq i \leq r$, the element $b_{p_i}$ is conjugate under $\Omega$ to
\begin{equation}\label{eq:gens}
\begin{cases}
(b_{p_{i-1}}, \id)\sigma & \text{if } i = 1 \text{ or } i = s, \\
(b_{p_{i-1}}, \id) & \text{otherwise}.
\end{cases}
\end{equation}
\end{proposition}

The group $G^{\text{geom}}$ is topologically generated by $b_{p_1}, \ldots, b_{p_r}$, and these elements satisfy the conjugacy relations given in Proposition~\ref{Pink:gen}. We show that the conjugacy class of the geometric IMG of $f$ is determined by any set of generators satisfying these relations.

Define elements $a_1, \ldots, a_r \in \Omega$ recursively by:
\begin{align}\label{generators}
\begin{split}
a_1 &= (a_r, \id)\sigma, \\
a_i &= (\id, a_{i-1}) \quad \text{for } 2 \leq i \leq s - 1, \\
a_s &= (\id, a_{s-1})\sigma, \\
a_i &= (a_{i-1}, \id) \quad \text{for } s + 1 \leq i \leq r.
\end{split}
\end{align}

Let $G \subset \Omega$ be the closed subgroup topologically generated by the elements $a_1, \ldots, a_r$, i.e.,
\[
G = \customlangle\customlangle a_1, \ldots, a_r \customrangle\customrangle.
\]
Since $G$ is a closed subgroup, it is naturally a profinite group and 
\[ G \simeq \lim_{n} G_n. \]
\begin{theorem}\label{thm:geom}
Let $k$ be a number field, and let $f(x) \in k(x)$ be a quadratic rational function with post-critical set $P$ as described in \eqref{pcorbit}. Then there exists $\gamma \in \Omega$ such that
\[
G^{\text{geom}}(f) = \gamma G \gamma^{-1}.
\]
\end{theorem}

We know that $G^{\text{geom}}$ is topologically generated by $b_{p_1}, \ldots, b_{p_r}$, and that some product of these elements is the identity. Corollary~\ref{cor:conja} shows that $a_i \sim b_{p_i}$ for all $1 \leq i \leq r$, and Proposition~\ref{prop:conj} shows that there exists $\beta \in \Omega$ such that
\[
\beta^{-1} G^{\text{geom}}(f) \beta \subset G.
\]
By \cite[Lemma~1.3.2]{Pinkpolyn}, it then follows that $G^{\text{geom}}(f) = \beta G \beta^{-1}$. We prove Proposition~\ref{prop:conj} and Corollary~\ref{cor:conja} in the next section.

\subsection{Proof of Theorem~\ref{thm:geom}} \ \\
If an automorphism $\gamma \in \Omega$ is conjugate to another automorphism $\gamma' \in \Omega$, we denote this by $\gamma \sim \gamma'$.

\begin{lemma}\label{lem:conj} \hfill
\begin{enumerate}
\item Let $\gamma = (\gamma_0, \gamma_1)\sigma \in \Omega$. The conjugacy class of $\gamma$ is
\[
\left\{ (\beta_0, \beta_1)\sigma \in \Omega \mid \gamma_0 \gamma_1 \sim \beta_0 \beta_1 \right\}.
\]

\item Let $\gamma = (\gamma_0, \gamma_1) \in \Omega$. The conjugacy class of $\gamma$ is
\[
\left\{ (\beta_0, \beta_1) \in \Omega \mid \gamma_0 \sim \beta_0 \text{ and } \gamma_1 \sim \beta_1 \right\}
\cup
\left\{ (\beta_0, \beta_1) \in \Omega \mid \gamma_0 \sim \beta_1 \text{ and } \gamma_1 \sim \beta_0 \right\}.
\]
\end{enumerate}
\end{lemma}

\begin{proof}
Let $\gamma = (\gamma_0, \gamma_1)\sigma$ and let $\alpha = (\alpha_0, \alpha_1)\tau$. Then:
\begin{align*}
\alpha \gamma \alpha^{-1}
&= (\alpha_0, \alpha_1)\tau (\gamma_0, \gamma_1)\sigma (\alpha_{(0)\tau}^{-1}, \alpha_{(1)\tau}^{-1})\tau^{-1} \\
&= (\alpha_0, \alpha_1)(\gamma_{(0)\tau}, \gamma_{(1)\tau}) \tau \sigma (\alpha_{(0)\tau}^{-1}, \alpha_{(1)\tau}^{-1}) \tau^{-1} \\
&= (\alpha_0, \alpha_1)(\gamma_{(0)\tau}, \gamma_{(1)\tau}) (\alpha_1^{-1}, \alpha_0^{-1}) \tau \sigma \tau^{-1} \\
&= (\alpha_0 \gamma_{(0)\tau} \alpha_1^{-1}, \alpha_1 \gamma_{(1)\tau} \alpha_0^{-1}) \sigma,
\end{align*}
since $\tau \sigma \tau^{-1} = \sigma$ in $S_2$. Letting $(\beta_0, \beta_1)\sigma = \alpha \gamma \alpha^{-1}$, we find
\[
\beta_0 \beta_1 = \alpha_0 \gamma_{(0)\tau} \gamma_{(1)\tau} \alpha_0^{-1}.
\]
Note also that
\[
\gamma_0 \gamma_1 = \gamma_1^{-1} (\gamma_1 \gamma_0) \gamma_1,
\]
so the statement follows.

Now for the second part, let $\gamma = (\gamma_0, \gamma_1)$ and $\alpha = (\alpha_0, \alpha_1)\tau$. Then:
\begin{align*}
\alpha \gamma \alpha^{-1}
&= (\alpha_0, \alpha_1)\tau (\gamma_0, \gamma_1) \tau^{-1} (\alpha_0^{-1}, \alpha_1^{-1}) \\
&= (\alpha_0 \gamma_{(0)\tau} \alpha_0^{-1}, \alpha_1 \gamma_{(1)\tau} \alpha_1^{-1}),
\end{align*}
which completes the proof.
\end{proof}

Corollary~\ref{cor:conja}, together with Proposition~\ref{Pink:gen}, justifies our choice of the set of generators $a_i$ for the group $G$.

\begin{corollary}\label{cor:conja}
For $1 \leq i \leq r$, the automorphism $a_i$ is conjugate to $b_{p_i}$ in $\Omega$.
\end{corollary}

\begin{proof}
The statement follows directly from Lemma~\ref{lem:conj} and Proposition~\ref{Pink:gen}.
\end{proof}

\begin{lemma}\label{b-rel}
The product $a_1 a_2 \ldots a_r$ is equal to the identity in $\Omega$.
\end{lemma}

\begin{proof}
It suffices to show that $a_1 \cdots a_r \restr{T_n} = \id$ for all $n \geq 1$. We compute:
\begin{align*}
a_1 \cdots a_r \restr{T_n}
&= (a_r, \id)\sigma (\id, a_1) \cdots (\id, a_{s-1})\sigma (a_s, \id) \cdots (a_{r-1}, \id) \restr{T_n} \\
&= \left(a_r a_1 \cdots a_{r-1} \restr{T_{n-1}}, \id\right).
\end{align*}
The result follows by induction on $n$ and taking the limit as $n \to \infty$.
\end{proof}

We now show that the image of the map $G \to \Omega$ given by $g \mapsto (g, g)$ is contained in $G$.

\begin{lemma}\label{(g,g)}
For all $g \in G$, the automorphism $(g, g)$ belongs to $G$.
\end{lemma}

\begin{proof}
Observe that $a_1^2 = (a_r, a_r)$ and $a_s^2 = (a_{s-1}, a_{s-1})$. For $1 < i < s$, we compute:
\begin{align*}
a_i (a_s a_i a_s^{-1})
&= (\id, a_{i-1})(\id, a_{s-1}) \sigma (\id, a_{i-1})(a_{s-1}^{-1}, \id)\sigma \\
&= (a_{i-1}, a_{i-1}) \in G.
\end{align*}

Similarly, for $s < j \leq r$, we have:
\begin{align*}
a_j (a_1 a_j a_1^{-1})
&= (a_{j-1}, \id)(a_r, \id)\sigma (a_{j-1}, \id)(\id, a_r^{-1})\sigma \\
&= (a_{j-1}, a_{j-1}) \in G.
\end{align*}

Since $a_1, \ldots, a_r$ are topological generators of $G$, it follows that for any $g \in G$, the element $(g, g)$ lies in $G$.
\end{proof}

Finally, the recursive definition of $a_1, \ldots, a_r$ implies that all sections of the generators also lie in $G$. Thus, $G$ is a self-similar subgroup of $\Omega$.

\subsection{Conjugacy of the Generators}

The next proposition defines the profinite geometric IMG up to conjugacy. Its proof follows the style of \cite[Proposition~2.4.1]{Pinkpolyn}, with the additional assumption that some product of the generators is the identity.

\begin{proposition} \label{prop:conj}
Let $n \geq 1$. Suppose that $b_i \in \Omega_n$ are elements such that $b_i$ is conjugate to $a_i \restr{T_n}$ for all $1 \leq i \leq r$, and assume that
\[
b_{(1)\tau} \cdots b_{(r)\tau} = \id
\]
for some permutation $\tau \in S_r$. Then there exists $\phi \in \Omega_n$ and $x_i \in G_n$ such that
\[
b_i = (\phi x_i)(a_i \restr{T_n})(\phi x_i)^{-1} \quad \text{for all } i.
\]
\end{proposition}

We denote this statement by $*_n$.

\begin{proof}
We proceed by induction on $n$.

The base case $n = 1$ is trivial. Assume $n \geq 2$, and that $*_{n-1}$ holds.

\textbf{Step 1:} Construct $c_1, \ldots, c_r \in \Omega_{n-1}$ such that $c_i \sim a_i \restr{T_{n-1}}$ and some product of the $c_i$ equals the identity.

Since $b_1$ is conjugate to $a_1 \restr{T_n}$, there exists $\gamma \in \Omega_n$ such that $b_1 = \gamma a_1 \gamma^{-1}$. If $\gamma \restr{T_1} \neq \id$, replace $\gamma$ with $\gamma' := \gamma a_1$. Then $\gamma' \restr{T_1} = \id$, and $b_1 = \gamma' a_1 \gamma'^{-1}$ still holds. So we may assume $\gamma = (\gamma_0, \gamma_1)$, and compute:
\begin{equation}\label{b1}
\begin{split}
b_1 &= \gamma a_1 \gamma^{-1} = (\gamma_0, \gamma_1)(a_r, \id)\sigma(\gamma_0^{-1}, \gamma_1^{-1}) \\
    &= (\gamma_0 a_r \gamma_1^{-1}, \gamma_1 \gamma_0^{-1})\sigma.
\end{split}
\end{equation}

Set
\[
c_r := \gamma_0 a_r \gamma_0^{-1} \in \Omega_{n-1}.
\]

We now organize the other $b_i$ according to their form. For $i \neq 1, s$, by Lemma~\ref{lem:conj}, $b_i$ must be of the form either $(d_{i-1}, \id)$ or $(\id, d_{i-1})$, with $d_{i-1} \sim a_{i-1} \restr{T_{n-1}}$.

Let
\[
S := \{2, \ldots, s-1\}, \quad T := \{s+1, \ldots, r\},
\]
and define subsets
\[
S_1 := \left\{ i \in S \mid b_i = (d_{i-1}, \id) \right\}, \quad T_1 := \left\{ i \in T \mid b_i = (\id, d_{i-1}) \right\}.
\]
Let $S_2 := S \setminus S_1$ and $T_2 := T \setminus T_1$.

Define $c_{i-1} \in \Omega_{n-1}$ for $i \neq 1, s$ by:
\begin{equation}\label{ci}
c_{i-1} := 
\begin{cases}
(\gamma_0 \gamma_1^{-1}) d_{i-1} (\gamma_0 \gamma_1^{-1})^{-1} & \text{if } i \in T_1 \cup S_2, \\
d_{i-1} & \text{if } i \in S_1 \cup T_2.
\end{cases}
\end{equation}

Assume without loss of generality that $\tau = \id$ (the general case is similar). Using the identity $b_1\ldots b_r=\id$ and \eqref{ci}, we can write 
\begin{align}\label{cs} 
\begin{split}
b_s^{-1}&=b_{s+1}\ldots b_rb_1\ldots b_{s-1}  \\
            &= \Big(\prod_{i \in T_2}c_{i-1}c_r \prod_{j \in S_2}c_{j-1}(\gamma_0\gamma_1^{-1}), (\gamma_0\gamma_1^{-1})^{-1}\prod_{k\in T_1}c_{k-1}\prod_{m \in S_1}c_{m-1}\Big)\sigma \\
\end{split}
\end{align}

Using the defined $c_i$ and \eqref{ci}, we set:
\begin{equation}\label{cs-1}
c_{s-1} := \left( \prod_{i \in T_2} c_{i-1} \cdot c_r \cdot \prod_{j \in S_2} c_{j-1} \cdot \prod_{k \in T_1} c_{k-1} \cdot \prod_{m \in S_1} c_{m-1} \right)^{-1}.
\end{equation}

This ensures that the product of $c_1, \dots ,c_r $ in some order equals identity. Since $b_s$ is conjugate to $a_s \restr{T_n}$, Lemma~\ref{lem:conj} implies that $c_{s-1} \sim a_{s-1} \restr{T_{n-1}}$.

By the induction hypothesis $*_{n-1}$, there exist $\beta \in \Omega_{n-1}$ and $x_i \in G_{n-1}$ such that:
\begin{equation}\label{ind}
c_i = (\beta x_i) a_i \restr{T_{n-1}} (\beta x_i)^{-1}.
\end{equation}

\textbf{Step 2:} Lift the conjugacies from level $n-1$ to level $n$.

Define $\phi = (\phi_0, \phi_1) \in \Omega_n$ by:
\[
\phi_0 := \beta, \quad \phi_1 := \gamma_1 \gamma_0^{-1} \beta.
\]

We now verify that $\phi$ conjugates each $b_i$ to $a_i$ up to an inner automorphism from $G_n$.

**Case 1: $i = 1$.**

Using \eqref{b1} and \eqref{ind}:
\[
\phi^{-1} b_1 \phi = (\beta^{-1} c_r \beta, \id) \sigma = (x_r a_r x_r^{-1}, \id) \sigma = (x_r, x_r) a_1 (x_r, x_r)^{-1}.
\]

By Lemma~\ref{(g,g)}, $y_1 := (x_r, x_r) \in G_n$ .

**Case 2:  $i \neq 1,s$.**

Following similar calculations using \eqref{ci} and \eqref{ind}, set:
\[
y_i := (x_{i-1}, x_{i-1}) \in G_n.
\]
A straightforward calculation shows that
\[
b_i = (\phi y_i) a_i (\phi y_i)^{-1}, \quad y_i := (x_{i-1}, x_{i-1}) \in G_n.
\]

**Case 3: $i = s$.**

We are left to find $y_s \in G_n$ such that
\[
\phi^{-1} b_s \phi = y_s a_s y_s^{-1} \restr{T_n}.
\]
Since $b_s^{-1} = b_{s+1} \cdots b_r \cdot b_1 \cdots b_{s-1}$, and $\phi^{-1} b_i \phi = y_i a_i y_i^{-1} \in G_n$ for all $i \neq s$, it follows that $\phi^{-1} b_s^{-1} \phi \in G_n$.

Suppose
\[
\phi^{-1} b_s^{-1} \phi = (g_0, g_1) \sigma \quad \text{for some } g_0, g_1 \in G_{n-1}.
\]
Then
\begin{equation}\label{g,h}
b_s^{-1} = (\phi_0 g_0 \phi_1^{-1}, \phi_1 g_1 \phi_0^{-1}) \sigma.
\end{equation}

From Equation~\eqref{cs}, we know
\begin{align*}
g_1 &= \phi_1^{-1} (\gamma_0 \gamma_1^{-1})^{-1} \prod_{k \in T_1} c_{k-1} \prod_{m \in S_1} c_{m-1} \phi_0 \\
&= \beta^{-1} \left( \prod_{k \in T_1} c_{k-1} \prod_{m \in S_1} c_{m-1} \right) \beta \\
&= \left(\prod_{k \in T_1} x_{k-1} a_{k-1} x_{k-1}^{-1} \right) \cdot \left(\prod_{m \in S_1} x_{m-1} a_{m-1} x_{m-1}^{-1}\right).
\end{align*}

Moreover,
\begin{align*}
c_{s-1} &= (\phi_0 g_0 \phi_1^{-1} \cdot \phi_1 g_1 \phi_0^{-1})^{-1} \\
&= \beta (g_0 g_1)^{-1} \beta^{-1},
\end{align*}
so using Equation~\eqref{ind} for $i = s-1$, we get
\begin{equation}\label{eq:g0g1}
g_0 g_1 = x_{s-1} a_{s-1}^{-1} x_{s-1}^{-1}.
\end{equation}

We now decompose $(g_0, g_1) \sigma$ as follows:
\begin{equation}\label{g01}
(g_0, g_1) \sigma = (\id, g_1)(g_0 g_1, \id) \sigma (\id, g_1^{-1}).
\end{equation}

We next show that $(\id, g_1) \in G_n$. Since for $i \in T_1$, $a_i = (a_{i-1}, \id)$, and for $j \in S_1$, $a_j = (\id, a_{j-1})$, we compute:
\begin{align*}
(g_1, \id)
&= \left(\prod_{i \in T_1} (x_{i-1} a_{i-1} x_{i-1}^{-1}, \id)\right) \cdot\left( \prod_{j \in S_1} (x_{j-1} a_{j-1} x_{j-1}^{-1}, \id) \right)\\
%&= \prod_{i \in T_1} (x_{i-1}, x_{i-1})(a_{i-1}, \id)(x_{i-1}^{-1}, x_{i-1}^{-1}) \quad \cdot \prod_{j \in S_1} (x_{j-1}, x_{j-1})(a_{j-1}, \id)(x_{j-1}^{-1}, x_{j-1}^{-1}) \\
&= \left(\prod_{i \in T_1} (y_i a_i y_i^{-1} )\right)\cdot a_s \cdot \left(  \prod_{j \in S_1} (y_j a_j y_j^{-1} a_s^{-1} )\right)\in G_n.
\end{align*}

Conjugating $(g_1, \id)$ by $a_1 \restr{T_n}$, we conclude that $(\id, g_1) \in G_n$. By equation~\ref{g01}, it is enough to show that $(g_0g_1,\id)\sigma$ is conjugate to $a_s^{-1}$ by an element of $G$.
Since $g_0g_1=x_{s-1}a_{s-1}^{-1}x_{s-1}^{-1}$, we find that element easily.
\[ (g_0g_1,\id)\sigma=(x_{s-1},x_{s-1})a_s^{-1}(x_{s-1}^{-1}, x_{s-1}^{-1}). \]
 
This completes the verification for $i = s$, and thus the proof of the proposition.
			
\end{proof}

 \section{Densely Settled Property of $G$} \label{sec:dsettled}

\begin{lemma}\label{odo:perm}
Let $g_1, \ldots, g_k$ be automorphisms of $T$. The product $g_1 \cdots g_k$ is an odometer if and only if $g_{\tau(1)}^{\ell_1} \cdots g_{\tau(k)}^{\ell_k}$ is an odometer for any permutation $\tau \in S_k$ and any $\ell_1, \ldots, \ell_k \in \ZZ_2^\times$.
\end{lemma}

\begin{proof}
By Proposition~\ref{pink:odometer}, $g_1 \cdots g_k$ is an odometer if and only if $\sgn_n(g_1 \cdots g_k) = -1$ for all $n \geq 1$. Since the image of $\sgn_n$ lies in the abelian group $\{\pm1\}$, permuting the product does not affect the sign. 

Moreover, for any $\ell \in \ZZ_2^\times$, we have $\ell \equiv 1 \pmod{2}$, and the continuity of $\sgn_n$ implies that
\[ 
\sgn_n(g^\ell) = \sgn_n(g)^\ell = \sgn_n(g),
\]
for all $g \in \Omega$ and $n \geq 1$. The result then follows from Proposition~\ref{pink:odometer}.
\end{proof}

The formula in Equation~\eqref{sgn:sigma} can be used to compute the signs of recursively defined elements. We now compute the signs of the generators of $G$.

\begin{lemma}\label{sgn:gen}
For $n \geq 1$, we have
\[
\sgn_n(a_i) =
\begin{cases}
-1 & \text{if } n \equiv i \text{ or } i + 1 - s \pmod{r}, \\
1 & \text{otherwise,}
\end{cases}
\]
for all $i = 1, \ldots, r$.
\end{lemma}

\begin{proof}
Let $\gamma = (\gamma_0, \gamma_1) \tau \in \Omega$ and fix $1 \leq i \leq r$. By Equation~\eqref{sgn:sigma} and the homomorphism property of $\sgn_n$, we have
\[ 
\sgn_n(\gamma) = \sgn_{n-1}(\gamma_0) \cdot \sgn_{n-1}(\gamma_1), \quad \text{for all } n > 1.
\]

Using this relation recursively, we find:
\[ 
\sgn_n(a_i) = \sgn_1(a_{i - n + 1}).
\]

From Equation~\eqref{sgn:sigma}, we know that $\sgn_1(a_j) = -1$ if and only if $j \equiv 1 \pmod{r}$ or $j \equiv s \pmod{r}$. Therefore,
\[ 
\sgn_n(a_i) = -1 \iff i - n + 1 \equiv 1 \text{ or } s \pmod{r},
\]
which is equivalent to the condition that $n \equiv i$ or $n \equiv i + 1 - s \pmod{r}$. The result follows.
\end{proof}

When $f(x)$ is a polynomial, the inertia generator at infinity gives an odometer in the geometric IMG of $f$, since the ramification index equals the degree at infinity. However, this is not guaranteed for non-polynomial maps. We now show that the existence of odometers in $G$ depends on the parity of $r$ and $s$, together with a divisibility condition.

For any $m \geq 1$, define the map
\begin{align*}
\sgn^m \colon \Omega &\to \{\pm 1\}^m \\
\alpha &\mapsto (\sgn_j(\alpha\restr{T_j}))_{j=1}^m.
\end{align*}

The formulas given by Lemma~\ref{sgn:gen} show that $\sgn_n(a_i)$ are determined modulo $r$. To determine whether an element $\gamma \in G$ is an odometer, it is enough to check whether $\sgn^r(\gamma) = (-1, \ldots, -1)$ in $\{\pm 1\}^r$.
\begin{theorem}\label{odoG}
 The group $G$ contains odometers if and only if either $r$ and $s$ are both even or $s$ is odd and $\frac{r}{\gcd(r,s-1)}$ is even.
\end{theorem}
\begin{proof}

Note that $\sgn_n(\gamma)$ depends only on the action of $\gamma$ on level $n$, and since $G_n$ is generated by $a_1\restr{T_n}, \ldots, a_n\restr{T_n}$, we may assume that $\gamma = a_1^{\ell_1} \cdots a_r^{\ell_r}$ for some $\ell_i \in \ZZ_2$.

Assume first that $r$ is odd and let $\gamma \in \Omega$ be an odometer. Then $\sgn_n(\gamma) = -1$ for all $n \geq 1$ by Proposition~\ref{pink:odometer}. But by Lemma~\ref{sgn:gen}, each $\sgn^r(a_i)$ has exactly two entries equal to $-1$. Any product of such vectors in $\{\pm 1\}^r$ will then have an even number of entries equal to $-1$, so $(-1, \ldots, -1)$ cannot be obtained. Hence $G$ contains no odometer if $r$ is odd.

Now assume $r$ is even. We consider two cases based on the parity of $s$.

\medskip
\noindent\textbf{Case 1:} $s$ even. We claim that $\gamma = a_1 a_3 \cdots a_{r-1}$ is an odometer. 
We will prove by induction on $n$ that $\gamma\restr{T_n}$ has order $2^n$. We compute:
\begin{align*}
\gamma^2 &= \left[(a_r, \id) \sigma (\id, a_2) \cdots (\id, a_{s-2})(a_s, \id) \cdots (a_{r-2}, \id)\right]^2 \\
&= [(a_r a_2 \cdots a_{s-2}, a_s \cdots a_{r-2}) \sigma]^2 \\
&= (a_r a_2 \cdots a_{s-2} a_s \cdots a_{r-2}, a_s \cdots a_{r-2} a_r a_2 \cdots a_{s-2}).
\end{align*}
Note that for any $\alpha,\beta$ in a group, $\alpha\beta$ and $\beta \alpha$ are conjugate and hence they have the same order. 
Thus, the order of $\gamma\restr{T_n}$ is twice the order of $a_r a_2 \cdots a_{r-2} \restr{T_{n-1}}$.

Next, we compute:
\begin{align*}
(a_r a_2 \cdots a_{r-2})^2 \restr{T_{n-1}} &= \left[(a_{r-1}, \id) \cdots (a_3, \id) \cdots (a_{r-1}, \id) \sigma\right]^2 \\
&= (a_1 a_3 \cdots a_{r-1}, a_1 a_3 \cdots a_{r-1}) \\
&= (\gamma\restr{T_{n-2}}, \gamma\restr{T_{n-2}}).
\end{align*}
So the order of $\gamma\restr{T_n}$ is 4 times that of $\gamma\restr{T_{n-2}}$. Since $\gamma\restr{T_1}$ has order 2, the claim follows by induction.

\medskip
\noindent\textbf{Case 2:} $s$ odd.

Let $\gamma \in G$ be an odometer. Then $\sgn^r(\gamma) = (-1, \ldots, -1)$. Again, we may assume that $\gamma$ is a product of some subset $S \subset \{1, \ldots, r\}$ of the generators:
\[ \gamma = \prod_{i \in S} a_i. \]
Let $S = \{m_1, \ldots, m_k\}$ and define $v_i := \sgn^r(a_{m_i})$. Each $v_i$ has exactly two $-1$ entries. So at least $r/2$ such vectors are needed to produce $(-1, \ldots, -1)$. On the other hand, since $\prod_{i=1}^r a_i = \id$, we have $\sgn^r(\gamma) = \sgn^r(\prod_{i \notin S} a_i)$, which implies $|S| \leq r/2$. Hence, $|S| = r/2$.

Next, each $v_{j}$ has $-1$ entries in positions $m_j$ and $m_j + 1 - s \mod r$. Since the collection $\{v_{j}\}_{j=1}^{r/2}$ yields $(-1, \ldots, -1)$ and the sets $\{m_j, m_j + 1 - s\}$ are disjoint, it follows that the union of these sets is $\{1, \ldots, r\}$.

Define the permutation $\tau\colon i \mapsto i + 1 - s \mod r$ in $S_r$. Its orbits have length $\ell = \frac{r}{\gcd(r, s-1)}$. Since the sets $\{m_i, m_i + 1 - s\}$ are pairwise disjoint and contained in the orbits of $\tau$, each such orbit must be of even length. Thus, $\ell$ must be even.

Conversely, suppose $s$ is odd and $\ell = \frac{r}{\gcd(r, s-1)}$ is even. Let $\tau$ be as above. In each orbit of $\tau$ (which has even length $\ell$), select every other element starting with the first, i.e., $m_1, m_3, \ldots, m_{\ell - 1}$. Let $T$ be the union of all such selections over each orbit.

By construction, $\{1, \ldots, r\}$ is the disjoint union of the pairs $\{i, i + 1 - s\}$ for $i \in T$. Since each $a_j$ has exactly two $-1$ signs and these locations are distinct for different $j \in T$, we obtain:
\[ \prod_{j \in T} \sgn^r(a_j) = (-1, \ldots, -1), \]
so $\prod_{j \in T} a_j$ is an odometer.

 \end{proof}
\begin{example}
Let $r = 8$ and $s = 7$. Consider the permutation $\tau \colon i \mapsto i + 2 \pmod{8}$ on the set $\{1, \ldots, 8\}$. Then $\tau$ can be written as the product of disjoint cycles:
\[
\tau = (1\ 3\ 5\ 7)(2\ 4\ 6\ 8)
\]
in the symmetric group $S_8$.

Let $\gamma = a_1 a_5 a_2 a_6 \in G$. We compute the sign vectors:
\[
\sgn^8(a_1) = (-1, 1, -1, 1, 1, 1, 1, 1), \quad
\sgn^8(a_2) = (1, -1, 1, -1, 1, 1, 1, 1),
\]
\[
\sgn^8(a_5) = (1, 1, 1, 1, -1, 1, -1, 1), \quad
\sgn^8(a_6) = (1, 1, 1, 1, 1, -1, 1, -1).
\]
Taking the product of these vectors coordinate-wise, we find:
\[
\sgn^8(\gamma) = (-1, -1, -1, -1, -1, -1, -1, -1).
\]
Hence, $\gamma$ is an odometer.
\end{example}
    
\subsection{Settled Elements in $G$}\label{def:settled} \ \\

In this section, we study the settled elements in $G$, motivated by the conjecture of Boston and Jones \cite{BJ-dedicata, BJ-quarterly, BJ-pams} and a question of Jones and Levy \cite{JonesLevy}. Boston and Jones introduced the notion of settledness to describe the behavior of Frobenius elements in arboreal representations of quadratic polynomials.

\begin{definition}\label{defn:settled}
Let $\gamma$ be an automorphism of the tree $T$. Let $n \geq 1$, and let $v$ be a vertex of $T$ at level $n$. We say $v$ is in a \emph{stable cycle of length} $k \geq 1$ of $\gamma$ if:
\begin{enumerate}
\item the orbit $\{(v)\gamma^i \mid i\geq 0\}$ has $k$ distinct elements.
\item for every $m > n$, all vertices in $T_m$ lying above $\{(v)\gamma^j \mid j \geq 0\}$ are in the same cycle of length $2^{m-n}k$.
\end{enumerate}

We say $\gamma \in \Omega$ is \emph{settled} if
\[
\lim_{n \to \infty} \frac{|\{ v \in V_n \mid v \text{ is in a stable cycle of } \gamma \}|}{|V_n|} = 1.
\]

A subgroup $H$ of $\Omega$ is called \emph{densely settled} if the set of settled elements is dense in $H$ under the profinite topology.
\end{definition}

Let $\gamma$ be an odometer. Since $\gamma\restr{T_n}$ is a $2^n$-cycle at each level $n$, every vertex $v$ of $T$ is in a stable cycle of $\gamma$, and hence $\gamma$ is settled.

More examples of settled elements can be constructed recursively. The following example is from \cite{CortezLukina}: consider the recursively defined element $\beta = (\beta, \gamma)$, where $\gamma = (\gamma, \id) \sigma$. The automorphism $\beta$ is settled and has exactly one fixed point. See \cite{CortezLukina} for additional examples of this type.

Roughly speaking, a settled element is formed by a collection of odometers in all levels of the tree. We formalize this observation in the following result.

\begin{proposition}\label{sodo}
Let $\gamma \in \Omega$. Assume that the vertex $w \in V_n$ is in a stable cycle of $\gamma$ of length $k \geq 1$. Then the section of $\gamma^k$ at each vertex $(w)\gamma^i$, for $i = 0, \ldots, k-1$, is an odometer.
\end{proposition}
\begin{proof}
Let $\gamma$ be an automorphism in $\Omega$, and let $w$ be a vertex of the tree $T$. Recall that $\gamma_w$ denotes the section of $\gamma$ at $w$. For any word $v$, we have
\[(wv) \gamma =(w) \gamma \cdot (v)\gamma_w. \]
Re-applying $\gamma$, we get:
\[ (wv) \gamma^2= ((w)\gamma \cdot (v)\gamma_w)\gamma = (w)\gamma^2 \cdot ((v)\gamma_w) \gamma_{(w)\gamma}. \]

Let $w_i := (w)\gamma^i$ for $0 \leq i < k$. Inductively, we obtain:
\[ (wv) \gamma^k= (w)\gamma^k \cdot (v)(\gamma_w\circ \gamma_{w_1} \circ   \cdots \circ  \gamma_{w_{k-1}}). \]
Thus, the section of $\gamma^k$ at $w$ is given by:
\[ (\gamma^k)_w = \gamma_w\circ \gamma_{w_1} \circ   \cdots \circ  \gamma_{w_{k-1}}. \]

Similarly, for any $0 \leq i < k$, the section of $\gamma^k$ at $w_i$ is:
\[ (\gamma^k)_{w_i} = \gamma_{w_{i}} \circ \gamma_{w_{i+1}} \cdots \circ \gamma_{w_{k+i-1}}, \]
where indices are considered modulo $k$.

Now, assume $w \in V_n$ lies in a stable cycle of length $k$ under $\gamma$. Then $(w)\gamma^k = w$, for each $i$, the collection $\{w_i, w_{i+1}, \ldots, w_{k+i-1}\}$ is equal to $\{w_0, \ldots, w_{k-1}\}$.

By Lemma~\ref{odo:perm}, it suffices to show that the section of $\gamma^k$ at $w$ is an odometer.

Since $w$ is in a stable cycle of length $k$, by definition, for every $m \geq n$, the vertices at level $m$ lying above the orbit $\{(w)\gamma^j\}$ form a cycle of length $k \cdot 2^{m-n}$ under $\gamma$. Therefore, the orbit of $\gamma^k$ containing those vertices has length $2^{m-n}$, meaning that the action of $(\gamma^k)_w$ on the tree rooted at $w$ is a $2^{m-n}$-cycle.

Since this holds for all $m \geq n$, it follows that $(\gamma^k)_w$ is an odometer.
\end{proof}

\begin{corollary}\label{cor:sodo}
Let $H$ be a self-similar subgroup of $\Omega$. If $H$ contains an automorphism with a stable cycle, then it must contain an odometer. In particular, if $H$ contains a settled element, then it contains an odometer.
\end{corollary}

\begin{proof}
Assume $\gamma \in H$ has a stable cycle. By Proposition~\ref{sodo}, a section of $\gamma^k$ is an odometer. Since $H$ is a group, $\gamma^k \in H$, and because $H$ is self-similar, every section of $\gamma^k$ also lies in $H$. Therefore, $H$ contains an odometer. If $H$ contains a settled element, then it necessarily has a stable cycle, and the result follows.
\end{proof}

\begin{theorem}\label{odoN}
If the group $G = \customlangle\customlangle a_1, \ldots, a_r \customrangle\customrangle$ does not contain any odometers, then the normalizer $N_{\Omega}(G)$ of $G$ does not contain any odometers.
\end{theorem}

\begin{proof}
Let $N = N_{\Omega}(G)$, and suppose $\gamma \in N$ is an odometer. Therefore, 
$\gamma = (\gamma_0, \gamma_1) \sigma$ and 
  \[
  \gamma a_1 \gamma^{-1} = (\gamma_0, \gamma_1) \sigma (a_r, \id) \sigma (\gamma_1^{-1}, \gamma_0^{-1}) \sigma = (\gamma_0 \gamma_1^{-1}, \gamma_1 a_r \gamma_0^{-1}) \sigma \in G.
  \]
 Thus, the element $\gamma_0 \gamma_1^{-1}$ is a section of $\gamma a_1 \gamma^{-1}$, which must lie in $G$ by the self-similarity of $G$.

Now consider the sign:
\[ \sgn_{n+1}(\gamma) = \sgn_{n+1}((\gamma_0, \gamma_1)) \cdot \sgn_{n+1}(\tau). \]
Since $\sgn_{n+1}(\tau) = 1$ for all $n \geq 1$, we have:
\[ \sgn_{n+1}(\gamma) = \sgn_n(\gamma_0) \cdot \sgn_n(\gamma_1)=\sgn_n(\gamma_0\gamma_1). \]
Using that $\sgn_n$ is a homomorphism and that $\sgn_n(\gamma_1) = \sgn_n(\gamma_1^{-1})$, it follows that:
\[ \sgn_n(\gamma_0 \gamma_1) = \sgn_n(\gamma_0 \gamma_1^{-1}). \]

Thus,
\begin{align*}
\sgn^{r+1}(\gamma) &= (\sgn_1(\gamma), \sgn_2(\gamma), \ldots, \sgn_{r+1}(\gamma)) \\
&= (\sgn_1(\gamma), \sgn_1(\gamma_0 \gamma_1^{-1}), \sgn_2(\gamma_0 \gamma_1^{-1}), \ldots, \sgn_r(\gamma_0 \gamma_1^{-1})).
\end{align*}

Since $\gamma$ is an odometer, we have $\sgn_1(\gamma) = -1$ and $\sgn_j(\gamma_0 \gamma_1^{-1}) = -1$ for $1 \leq j \leq r$, implying that $\gamma_0 \gamma_1^{-1} \in G$ is an odometer — contradicting the assumption that $G$ contains no odometers. 
\end{proof}

\begin{theorem} \label{IMGfsettled}
Let $k$ be a number field and let $f(x) \in k(x)$ be a quadratic rational map whose post-critical set $P$ is described in \ref{pcorbit}. Assume that:
\begin{enumerate}
\item $r$ is odd, or
\item $r$ is even, $s$ is odd, and $\frac{r}{\gcd(r, s - 1)}$ is odd.
\end{enumerate}
Then the arithmetic and geometric IMGs of $f$ do not contain any settled elements. Moreover, for any $a \in \PP^1(k) \setminus P$, the arboreal representation of $(f, a)$ is not densely settled.
\end{theorem}

\begin{proof}
If $(r, s)$ satisfy one of the conditions in the statement, then by Theorem~\ref{odoG}, $G$ does not contain any odometers. By Theorem~\ref{odoN}, the normalizer of $G$ also does not contain any odometers. Since the arithmetic IMG of $f$ is contained in this normalizer, it cannot contain any settled elements. Similarly, the arboreal representation is a subgroup of the arithmetic IMG, and therefore it also does not contain any settled elements.
\end{proof}

\begin{corollary}\label{cor:arboreal}
Let $f(x) = \frac{1}{(x - 1)^2} \in \QQ(x)$.
\begin{enumerate}
\item The geometric and profinite arithmetic IMGs of $f$ do not contain any settled elements.
\item The arboreal Galois group $G_{f,a}$ does not contain any settled elements for any $a \in \QQ \setminus \{0, 1\}$.
\end{enumerate}
\end{corollary}

\begin{proof}
The rational function $f$ has two critical points, $1$ and $\infty$, and the post-critical set of $f$ is $P = \{0, 1, \infty\}$, with the orbit:
\[ 1 \to \infty \to 0 \to 1. \]
Thus, using the notation from \eqref{pcorbit}, we have $r = 3$ and $s = 2$. The conclusion then follows directly from Theorem~\ref{IMGfsettled}.
\end{proof}

\subsection{A question of Jones and Levy on settledness}\label{Jones-Levysec} \ \\
Let $f$ be a quadratic rational function over the finite field $k$, and let $a \in k$. Let $T_{f,a}$ be the infinite tree formed by the inverse image sets $f^{-n}(a)$ for all $n \geq 1$. For each $n \geq 1$, let $\mu_n$ be the probability measure assigning equal mass to each point, counting multiplicity. This yields a probability measure $\mu$ on $\delta T_{f,a} = \lim_{n \to \infty} f^{-n}(a)$ by
\[ \mu(\Sigma) = \lim_{n \to \infty} \mu_n(\pi_n(\Sigma)). \]

The pair $(f, a)$ is called \emph{settled} if the union of all open $G_k$-orbits on $\delta T_{f,a}$ has measure 1. In other words, for $\mu$-almost every $v \in \delta T_{f,a}$, the $G_k$-orbit of $v$ is open. Since $k$ is a finite field, $G_k$ is topologically generated by a single element, namely the Frobenius automorphism $\Frob_k$.

In \cite[Proposition~6.4]{CortezLukina}, Cortez and Lukina show that an automorphism $\gamma \in \Omega$ is settled if and only if for $\mu$-almost all $v \in \delta T$, the closure of the $\gamma$-orbit of $v$ is open. That is, $\gamma$ is settled if and only if the set $\{ \gamma^k(v) \mid k \in \ZZ_2 \}$ is open in $\delta T$ for $\mu$-almost all $v \in \delta T$.

Thus, we may reformulate the definition of Jones and Levy:
A pair $(f,a)$ is settled if $\Frob_k$ is settled in the automorphism group of the tree $T_{f,a}$.

\begin{question*}[Jones and Levy \cite{JonesLevy}]
Let $k$ be a finite field of characteristic $>2$. Must $(f, a)$ be settled for all $a \in \PP^1(k)$?
\end{question*}

Let $f(x)=\frac{1}{(x-1)^2}$ and let $p > 2$ be a prime. Let $\bar{a}$ denote the reduction of $a$ modulo $p$. Let $k_{\infty,a}$ be the field obtained by adjoining to $\QQ$ the roots of $f^{(n)}(x) - a$ for all $n \geq 1$. Let $\bar{f}$ denote $1/(x - 1)^2 \in \FF_p(x)$. 

If the pair $(\bar{f}, \bar{a})$ is settled, then the automorphism $\Frob_p$ is settled. If $p$ is unramified in $k_{\infty,a}$, then the Frobenius element for $p$ in $k_{\infty,a}$ is settled in the arboreal Galois group $G_{f,a}$, which contradicts Corollary~\ref{cor:arboreal}. This provides a negative answer to the question of Jones and Levy.

\begin{proposition}\label{prop:JL}
Let $f(x) = \frac{1}{(x - 1)^2} \in \ZZ(x)$ and let $p > 2$ be a prime that is unramified in $k_{\infty,a}$. Then the pair $\left(\frac{1}{(x - 1)^2}, a\right)$ is not settled for any $a \in \FF_p \setminus \{0, 1\}$.
\end{proposition}

\section{The structure of the group $G$}\label{sec:structure}
For $i \geq 1$, let $N_i$ be the topological closure of the subgroup of $G$ generated by all $G$-conjugates of
\begin{equation}\label{defnNi} 
\{ a_j \colon j \not\equiv i, i - s + 1 \text{ mod } r \}.
\end{equation}

Note that since $N_i$ is a closed subgroup of $G$, the quotient $G / N_i$ is complete under the profinite topology for any $i = 1, \ldots, r$. 

Moreover, $G_n / N_{i,n}$ is a cyclic group generated by $a_i \restr{T_n}$ for each $n \geq 1$, and hence $G / N_i$ is topologically generated by $a_i$. 

\begin{lemma}\label{ordGN}
Let $H$ be a closed subgroup of $G$ such that $G_n/H_n$ is cyclic for all $n\geq 1$. The order of $G_{n+1}/H_{n+1}$ can be at most $2$ times the order of $G_{n}/H_n$ for any $n\geq 1$.
\end{lemma}

\begin{proof}
Consider the composition of $\pi_n \colon G_{n+1}\to G_n$ and $\phi \colon G_n \to G_n/H_n$. Since the image of $H_{n+1}$ under $\pi_n$ is contained in $H_n$, we have a homomorphism
$$G_{n+1}/H_{n+1} \to G_n/H_n.$$
Since this map is surjective, it is enough to show that the size of its kernel is less than or equal to two.
The kernel is given by $\ker(G_{n+1} \to G_n)H_{n+1}/H_{n+1}$ and it is isomorphic to a quotient of $\ker(G_{n+1} \to G_n)$. Since 
 $\ker(G_{n+1} \to G_n)$ is an elementary abelian $2$-group, its quotient is also so. Since $G_n/H_n$ is cyclic for all $n\geq 1$, as a cyclic subgroup of an elementary abelian $2$ group, kernel can be of order at most two.
\end{proof}

We now prove a lemma on the structure of the subgroup $N_i$, which will play an essential role in several of the arguments that follow.

\begin{lemma}\label{Ni}
For $1 \leq i \leq r$, let $N_i$ be the subgroup of $G$ defined in ~\ref{defnNi}.
\begin{enumerate}
\item $N_s = N_{s-1} \times N_{s-1}$
\item For $i\neq s$, the subgroup $N_i \cap (\Omega \times \Omega)$ is a subset of $(N_{i-1} \times N_{i-1})\customlangle\customlangle (a_{i-1}, a_{i-1}^{-1}) \customrangle\customrangle$.
\item For $i \in \{1,\ldots,r\}$,
 if $(\gamma_0, \gamma_1) \in N_i$ for some $i \in \{1, \ldots, r\}$, then $\gamma_0\gamma_1 \in N_{i-1}$.
\end{enumerate}
\end{lemma}
\begin{proof}
Let $H$ be a normal subgroup of $G$, and assume $H$ contains an element $\gamma = (\gamma_0, \id)$ for some $\gamma_0 \in G$. Then $H$ also contains $(\id, \gamma_0)$, since \begin{align}\label{conjbya1}
\begin{split}
a_1 \gamma a_1^{-1} &= (a_r, \id) \sigma (\gamma_0, \id)(\id, a_r^{-1}) \sigma \\
&= (a_r a_r^{-1}, \gamma_0) = (\id, \gamma_0) \in H.
\end{split}
\end{align}

Similarly, if $H$ contains $(\id, \beta_1)$ for some $\beta_1 \in G$, then
\begin{align}\label{conjbyas}
\begin{split}
a_s \beta a_s^{-1} &= (\id, a_{s-1}) \sigma (\id, \beta_1)(a_{s-1}^{-1}, \id) \sigma \\
&= (\beta_1, \id)
\end{split}
\end{align}
is also in $H$. 

We first show that $N_{s-1} \times N_{s-1}$ is contained in $N_s$. From the definition, $a_i \restr{T_1} = \id$ for $i \neq 1, s$, hence their conjugates are also identity on level one. So the topological generators of $N_s$ lie in the kernel of $\pi_1 : \Omega \to \Omega_1$, which is an open (and closed) subset of the profinite group $\Omega$. Thus, $N_s \subset \ker(\pi_1)$.

By ~\ref{defnNi}, the normal subgroup $N_{s-1}$ is topologically generated by
\begin{equation}\label{genset(s-1)}
\{ \beta_0 a_{i-1} \beta_0^{-1} \mid \beta_0 \in G,\ i \neq 1, s \}.
\end{equation}

For $s < i \leq r$, we have $a_i = (a_{i-1}, \id) \in N_{s}$ and for $2 \leq i \leq s-1$, $a_i = (\id, a_{i-1}) \in N_{s}$. By \eqref{conjbyas}, $(a_{i-1}, \id) \in N_s$ for all $i \neq 1,s$. 

For $\beta_0 \in G$, we have $\beta = (\beta_0, \beta_0) \in G$, and for $1 \leq i \leq r$, $i \neq 1, s$,
\begin{align}\begin{split} (\beta_0 a_{i-1} \beta_0^{-1}, \id) &= (\beta_0, \beta_0)(a_{i-1}, \id)(\beta_0^{-1}, \beta_0^{-1}) \\
     &= \beta a_i \beta^{-1} \in N_s. \end{split} \end{align}
Since $N_s$ is closed, $N_{s-1} \times \{\id\} \subset N_s$, and by conjugating with $a_s$, we also have $\{\id\} \times N_{s-1} \subset N_s$. %Hence $N_s = N_{s-1} \times N_{s-1}$. 

Since for any $g\in G$ and $i\neq 1,s$, we have $ga_{i-1}g^{-1}$ is an element of $N_{s-1}$, we find that $N_{s-1} \times N_{s-1}$ contains the topological generators of $N_s$. Since $N_{s-1}\times N_{s-1}$ is a closed subgroup of $\Omega$, we obtain that $$N_s=N_{s-1} \times N_{s-1}.$$ 

To prove the second part of the statement,  assume $i \neq s$. Set $$M_i := (N_{i-1} \times N_{i-1}) \customlangle\customlangle (a_{i-1}, a_{i-1}^{-1}) \customrangle\customrangle.$$ We will show $M_i$ is a normal subgroup of $$\tilde{G} := (G \times G) \rtimes \customlangle \sigma \customrangle$$ with abelian quotient.

Since $N_{i-1}$ is normal in $G$, $N_{i-1} \times N_{i-1}$ is normal in $\tilde{G}$. Also, $G / N_j$ is generated by $a_j$, so we have a surjection
\[
 \varphi : (\customlangle\customlangle a_{i-1} \customrangle\customrangle \times \customlangle\customlangle a_{i-1} \customrangle\customrangle) \rtimes \customlangle \sigma \customrangle \to \tilde{G}/(N_{i-1} \times N_{i-1}). 
 \]
Let $H := \customlangle\customlangle (a_{i-1}, a_{i-1}^{-1}) \customrangle\customrangle$. By the third isomorphism theorem,
$$(\customlangle\customlangle a_{i-1} \customrangle\customrangle \times \customlangle\customlangle a_{i-1} \customrangle\customrangle) \rtimes \customlangle \sigma \customrangle /H \simeq \tilde{G}/M_i. $$ 
Since the group on the left is abelian, it follows that $\tilde{G}/M_i$ is abelian as well.

We claim that the image of $N_i \cap (\Omega \times \Omega)$ is trivial in $\tilde{G}/M_i$. Let $$S_i=\{ a_j \mid j\neq i,i+1-s\}.$$ The subgroup $N_i \cap (\Omega\times \Omega)$ contains the conjugates of the elements $ a_j$ for $ j\in S_ i \setminus \{1,s\}$. Since $i\neq s$, the set $S_i \cap \{1,s\}$ is not empty. For $j,k \in S_ i \setminus \{1,s\}$ (not necessarily distinct) and any $g, h \in G$ the product
$$ga_kg^{-1}ha_th^{-1}$$ is also in $N_i \cap (\Omega \times \Omega)$. In fact, these elements all together (topologically) generate the subgroup $N_i \cap (\Omega \times \Omega)$.

Since $\tilde{G}/M_i$ is abelian, the image of the described elements in the quotient $\tilde{G}/M_i$ are either identity, or $a_ka_t$ for $k,t \in S_i \cap \{1,s\}$. Therefore, it is enough to show that the images of such elements are identity in $\tilde{G}/M_i$. A quick computation of $a_1^2,a_s^2$ and $a_1a_s$ shows that the claim holds.

Third statement is a corollary of the first and second statements.
\end{proof}

\begin{proposition}\label{infquo}
For any $i=1,\ldots,r$, the quotient group $G/N_i$ is isomorphic to the group of $2$-adic integers $\ZZ_2$.
\end{proposition}
\begin{proof}
Since $N_i$ contains all $G$-conjugates of $a_j$ for $j \not\equiv i,i+s-1 \pmod{r} $, and $a_1 \cdots a_r=\id$, the quotient $G/N_i$ is topologically generated by the image of $a_i$, Hence, the closed subgroup generated by $a_i$ surjects onto $G/N_i$ for each $i=1,\ldots,r$. 

To show that $G/N_i$ is infinite and isomorphic to $\ZZ_2$, we estimate the growth of the index $[G_n \colon N_{i,n}]$ as $n$ increases. Using the recursive structure of the generators $a_j$, we obtain inequalities showing that the quotient grows by at least a factor of $2$ as level increases from $n$ to $n+r$.

Let $N_{i,n}$ denote the restriction of $N_i$ to level  $n$.  Assume $a_s^k\restr{T_n} \in N_{s,n}$ for some positive integer $k$. Since $N_s$ is contained in the stabilizer of level $1$, $k$ must be even. By Lemma~\ref{Ni}, we have $a_{s-1}^{k/2}\restr{T_{n-1}}$ is in $N_{s-1,n-1}$. This shows that $|G_n/N_{s,n}| \geq 2|G_{n-1}/N_{s-1,n-1}|$. 

For $2 \leq i <s$, we have $a_i^k=(\id, a_{i-1}^k)$ for any $k\geq 1$. Assume $a_i^k\restr{T_{n}} \in N_{i,n}$ for some positive integer $k$, then $a_{i-1}^k \restr{T_{n-1}} \in N_{i-1,n-1}$ by Lemma~\ref{Ni}. Now one obtains that $a_{1}^k\restr{T_{n}}$ is in $N_{1,n}$ by applying Lemma~\ref{Ni} consecutively. Hence $|G_n/N_{s,n}| \geq 2|G_{n-s+1}/N_{1,n-s+1}|$.

Assume that some power of $a_1\restr{T_{n-s+1}}$ is in $N_{1,n-s+1}$, that is $a_1^m\restr{T_{n-s+1}}=(a_r^{m/2} \restr{T_{n-s}},a_r^{m/2}\restr{T_{n-s}}) \in  N_{1,n-s+1}$ for some positive integer $m$. By Lemma~\ref{Ni}(3), we find that $a_r^m\restr{T_{n-s}} \in N_{r,n-s}$. To summarize, we have
$|G_n/N_{s,n}| \geq 2|G_{n-s}/N_{r,n-s}|$. 

For $s < i \leq r$, we apply Lemma~\ref{Ni} consecutively again to obtain 
 \begin{equation}\label{eq:bdquo} |G_n/N_{s,n}| \geq 2|G_{n-r}/N_{s,n-r}| 
 \end{equation} 
 for all $n\geq r$.
 
This shows that as $n \to \infty$, the order of $G_n/N_{s,n}$ goes to infinity. In fact, this discussion shows that $G/N_{i}$ is an infinite group for all $1\leq i \leq r$. By equation~\ref{eq:bdquo} and Lemma~\ref{ordGN}, we find that for any $n\geq 1$, there is some level $k_n \geq 1$ such that $G_{k_n}/ N_{i,k_n}$ is a cyclic group of order $2^n$. Since $N_i$ is a closed subgroup, $G/N_i$ is the inverse limit of $G_n/N_{i,n}$, and hence it is isomorphic to $\ZZ_2$.
\end{proof}

\begin{corollary}\label{intersection}
For any $i=1,\ldots,r$, the intersection of $N_i$ and $\customlangle\customlangle a_i \customrangle\customrangle$ is trivial. Hence \[ G= N_i \rtimes \customlangle\customlangle a_i \customrangle\customrangle. \]
\end{corollary}
\begin{proof}
Let $1\leq i \leq r$. The quotient map $\customlangle\customlangle a_i \customrangle\customrangle \to G/N_i$ is surjective. By Proposition~\ref{infquo}, $G/N_i$ is isomorphic to $\ZZ_2$. Since $\customlangle\customlangle a_i \customrangle\customrangle$ is a pro-cyclic group that surjects onto $\ZZ_2$, it is also isomorphic to $\ZZ_2$. By Lemma~\ref{surjZ2}, the $\ZZ_2$-linear map $\customlangle\customlangle a_i \customrangle\customrangle \to G/N_i$ is an isomorphism. 

Since $G/N_i \simeq \ZZ_2$ and the intersection of $N_i$ with the pro-cyclic subgroup generated by $a_i$ is trivial, we conclude that $G$ decomposes as a semi-direct product.
\end{proof}

\begin{lemma}\label{surjZ2}
Let $k\geq 1$ be an integer. Any surjective $\ZZ_2$-linear map from $\ZZ_2^k$ to itself is injective.
\end{lemma}

\begin{proof}
This is a special case of Theorem~2.4 of \cite{Matsumura}. 
\end{proof}
\begin{remark}
The normal subgroup $N_i$ contains all $a_j$ except for $j=i$ and $j=i-s+1$. Since $a_1\ldots a_r=\id$, we have $N_i=a_{i-s+1}a_iN_i$. Therefore, $a_iN_i=a_{i-s+1}^{-1}N_i$. We showed in Proposition~\ref{infquo} that 
$G/N_i \simeq \ZZ_2$ where the isomorphism is given by $a_i^kN_i \mapsto k$ for any $k\in \ZZ_2$.
\end{remark}

Using Lemma \ref{infquo} and a group theoretical proposition due to Bartholdi and Nekrashevych, we can show that $G$ is torsion-free.

\begin{proposition}\label{BN} \cite[Proposition 2.2]{BNIMG}
Let $H$ be a self-similar subgroup of $\Omega$ and let $N$ be a normal subgroup of $H$ contained in the stabilizer of the first level. If $H/N$ is torsion free, then so is $H$.
\end{proposition}

\begin{theorem}\label{tfree}
The group $G=\customlangle\customlangle a_1,\ldots, a_r \customrangle\customrangle$ described in \ref{generators} is torsion free.
\end{theorem}
\begin{proof}
By Proposition~\ref{infquo}, the quotient $G/N_s$ is torsion-free as it is isomorphic to $\ZZ_2$. Since $N_s$ is generated by $a_j=(a_{j-1},\id)$ for $s<i\leq r$ and $a_i=(\id, a_{i-1})$ for $1<i <s$, it is contained in the stabilizer of the first level of the tree. Therefore, by Proposition~\ref{BN}, the group $G$ is torsion-free.
\end{proof}

\section{ Maximal Abelian Quotient of $G$}\label{maq}
Let $[G,G]$ denote the topological closure of the subgroup of $G$ generated by all commutators $[g,h]$ for $g,h \in G$. The quotient $G/[G,G]$ is called the \emph{maximal abelian quotient} of $G$, and is denoted by $G_{\text{ab}}$. Let $N$ be the normalizer of $G$ in $\Omega$. Then $N$ acts on $G_{\text{ab}}$ by conjugation, and this action defines a homomorphism
\[
\varphi\colon N \to \Aut(G_{\text{ab}}).
\]
Since $G$ acts trivially on $G_{\text{ab}}$, $\varphi$ factors through a homomorphism
\[
\bar{\varphi}\colon N/G \to \Aut(G_{\text{ab}}).
\]

\begin{proposition}\label{thm:final}
Let $k$ be a number field. Let $f(x) \in k(x)$ be a PCF quadratic rational function. 
If the induced homomorphism $\bar{\varphi}\colon N(G)/\gm \to \Aut(G_{\text{ab}})$ is injective, then the field of constants associated with $f$ is contained in the cyclotomic field $k(\zeta_{2^n} \colon n \geq 1)$.
\end{proposition}

\begin{proof}
We have a surjection $G_k = \Gal(\bar{k}/k) \to \ga(f)/G$ and the quotient $\ga(f)/G \simeq \Gal(F/k)$, where $k$ is the field of constants associated with $f$. The action of $G_k$ on $G_{\text{ab}}$ factors through $\Gal(F/k)$ via the exact sequence
\[ 1 \to G \to \ga \to \Gal(F/k) \to 1. \]

The Galois group $G_k$ acts on the images of the inertia elements in $G_{\text{ab}}$ via the cyclotomic character; see, for example, \cite[Remark~4.7.5]{Szamuely_2009} or \cite[Lemma~58.13.5]{stacks-project}. 
Since $G_{\text{ab}}$ is generated by the images of the inertia elements, the action of $G_k$ on $G_{\text{ab}}$ is given by the cyclotomic character. 

Hence, the homomorphism $\Gal(F/k) \to \Aut(G_{\text{ab}})$ is given by the cyclotomic character. %In other words, a quotient of $\Gal(F/k)$ is isomorphic to the Galois group of a subfield of $k(\zeta_{2^n} \colon n \geq 1)$ over $k$.
and the injectivity of $\bar{\varphi}$ tells us that $F$ is contained in the cyclotomic field $k(\zeta_{2^n} \colon n \geq 1)$.
\end{proof}

\begin{remark}
\hfill
\begin{enumerate}
\item By results of Hamblen and Jones \cite{HamblenJones}, if $k$ contains the critical points of $f$, then $F$ contains the $2^n$th roots of unity for all $n \geq 1$. Therefore, in this case, $F = k(\zeta_{2^n} \colon n \geq 1)$.
\item The injectivity of $\bar{\varphi} \mid N/G \to \Aut(G_{\text{ab}})$ is essential in the proof above. Without it, the group $\ga(f)/\gm(f)$ may fail to be abelian.
\end{enumerate}
\end{remark}

Next, we describe a simple condition on the structure of the subgroup $[G,G]$ that ensures the injectivity of $\bar{\varphi}$.

\begin{proposition}\label{prop:inj}
Assume that for any $\gamma = (\gamma_0, \gamma_1) \in [G,G]$, the product $\gamma_0 \gamma_1$ is in $[G,G]$. Then the kernel of the homomorphism $\varphi \colon N \to \Aut(G_{\text{ab}})$ is equal to $G$.
\end{proposition}

\begin{proof}
Let $M := \ker(\varphi)$, so $M$ is a subgroup of $N$ and it acts trivially on $G_{\text{ab}}$. Since $G$ acts trivially on $G_{\text{ab}}$ by definition, we have $G \subset M$. Our goal is to prove that $M \subset G$, i.e., $M = G$.

Let $\gamma \in M$. We will prove by induction on $n \geq 1$ that $\gamma \restr{T_n} \in G_n$. 

We first reduce to the case where $\gamma$ stabilizes the first level. If $\gamma \restr{T_1} \neq \id$, then $\gamma a_1 \restr{T_1} = \id$ (since $a_1$ swaps the subtrees and $\gamma$ must reverse it). Replacing $\gamma$ by $\gamma a_1$ (which still lies in $M$ since $a_1 \in G \subset M$), we may assume that $\gamma = (\gamma_0, \gamma_1) \in N \cap (\Omega \times \Omega)$.

Let us examine the implications of $\gamma \in M$. This means that for all $i = 1, \dots, r$, the commutator
\[
\gamma a_i \gamma^{-1} a_i^{-1}
\]
lies in $[G,G]$. We will show that this implies that both $\gamma_0$ and $\gamma_1$ are also in $M$.

We handle the cases $i \neq 1, s$ first. For $2 \leq i \leq s-1$, we have $a_i = (\id, a_{i-1})$, and $a_s a_i a_s^{-1} = (a_{i-1}, \id)$. Then
\[
\gamma (a_{i-1}, \id) \gamma^{-1} (a_{i-1}^{-1}, \id) = (\gamma_0 a_{i-1} \gamma_0^{-1} a_{i-1}^{-1}, \id).
\]
By assumption on $\gamma \in M$, the left-hand side lies in $[G,G]$, and our hypothesis then implies that $\gamma_0 a_{i-1} \gamma_0^{-1} a_{i-1}^{-1} \in [G,G]$ for all $i \neq 1, s$.

For $i = s$, we compute
\[
\gamma a_s \gamma^{-1} a_s^{-1} = (\gamma_0, \gamma_1)(\id, a_{s-1}) \sigma (\gamma_0^{-1}, \gamma_1^{-1})(a_{s-1}^{-1}, \id)\sigma = (\gamma_0 \gamma_1^{-1}, \gamma_1 a_{s-1} \gamma_0^{-1} a_{s-1}^{-1}),
\]
which lies in $[G,G]$ since $\gamma \in M$. Again using the assumption on $[G,G]$, we deduce that
\[
\gamma_0 a_{s-1} \gamma_0^{-1} a_{s-1}^{-1} \in [G,G].
\]
Thus, $\gamma_0 \in M$.

A similar argument using the conjugates $b_i := a_1 a_i a_1^{-1} = (\id, a_{i-1})$ for $i > s$, shows that $\gamma_1 \in M$.

We now inductively show that $\gamma \restr{T_n} \in G_n$ for all $n$. The base case $n = 1$ is clear. Assume the inductive hypothesis holds for level $n-1$.

Since $\gamma_1 \in M$, there exists $g_0 \in G$ such that $\gamma_1 g_0^{-1} \restr{T_{n-1}} = \id$. Then, by Lemma~\ref{(g,g)}, the element $(g_0, g_0) \in G$, so replacing $\gamma$ by $\gamma (g_0^{-1}, g_0^{-1})$, we may assume that
\[
\gamma \restr{T_n} = (\gamma_0 \restr{T_{n-1}}, \id).
\]

We now use the fact that
\[
\gamma a_1 \gamma^{-1} a_1^{-1} = (\gamma_0 a_r \gamma_1^{-1} a_r^{-1}, \gamma_0^{-1} \gamma_1)
\]
lies in $[G,G]$. We claim that $[G,G]$ is a subgroup of $N_i$ for all $i \in \{1,\ldots, r \}$. This follows since each $G/N_i$ is abelian and $G/[G,G]$ is the maximal abelian quotient of $G$. 

Since $[G,G] \subset N_s = N_{s-1} \times N_{s-1}$, we conclude that $\gamma_0^{-1} \gamma_1$ lies in $N_{s-1}$. But $\gamma_1 = \id$ on level $n-1$, so $\gamma_0 \restr{T_{n-1}} \in N_{s-1, n-1}$.

It follows that $\gamma \restr{T_n} = (\gamma_0 \restr{T_{n-1}}, \id)$ lies in $N_{s,n}$, which is a subgroup of $G_n$. This completes the inductive step, and hence $\gamma \in G$.
\end{proof}

 \subsection{The case $s=2$, where the critical points of $f(x)$ are $p_r$ and $p_1$ }

Throughout this section, we assume that $s = 2$, so the generators
$a_1, \ldots, a_r \in \Omega$ are defined recursively by
\begin{align}\label{generators}
\begin{split}
a_1 &= (a_r, \id)\sigma, \\
a_2 &= (\id, a_1)\sigma, \\
a_i &= (a_{i-1}, \id) \quad \text{for } 3 \leq i \leq r.
\end{split}
\end{align}

For each $i$, the subgroup $N_i$ is the topological closure of the subgroup generated by the $G$-conjugates of the elements $a_j$ with $j \not\equiv i, i-1 \pmod{r}$. We remind the reader that the indices of both the subgroups $N_j$ and the generators $a_i$ are interpreted modulo~$r$. We now verify that the assumption in Proposition~\ref{prop:inj} holds in this case.

\begin{theorem}\label{Gabsurj}
For any $i \in \{1, \ldots, r\}$, the homomorphism
\[
G \to G/N_i \times G/N_{i+1} \times \cdots \times G/N_{i+r-2}
\]
is surjective.
\end{theorem}

\begin{proof}
For any $i \in \{1, \ldots, r\}$, the quotient group $G/N_i$ is topologically generated by the image of $a_i$. Fix $1 \leq i \leq r$. For each $j$ with $i \leq j \leq i + r - 2$, we construct an automorphism $\gamma_j \in G$ such that $\gamma_j N_j = a_j N_j$ and $\gamma_j N_k = N_k$ for all other $k$ in $\{i, \ldots, i + r - 2\}$ with $k \ne j$.

Set
\[
\gamma_j := \prod_{\ell = j}^{i + r - 2} a_\ell \in G.
\]
We claim this satisfies the desired properties. Since $N_j$ contains the generators $a_{j+1}, \ldots, a_{i + r - 2}$, it follows that $\gamma_j \equiv a_j \pmod{N_j}$. On the other hand, for any $k \ne j$, $N_k$ contains all $a_\ell$ except possibly $a_k$ and $a_{k-1}$ (modulo $r$). So either $\gamma_j \in N_k$, or $\gamma_j \equiv a_{k-1} a_k \pmod{N_k}$. Since the product $a_1 \cdots a_r = \id$, it follows that $a_{k-1} a_k \equiv \id \pmod{N_k}$, so $\gamma_j \equiv \id \pmod{N_k}$ as required.
\end{proof}

\begin{example}
Let $r = 4$ and $s = 2$. Then the map
\[
G \to G/N_1 \times G/N_2 \times G/N_3
\]
is surjective. Since each quotient $G/N_i \simeq \ZZ_2$ via the map $a_i \mapsto 1$, we can describe the induced homomorphism $G \to \ZZ_2^3$ explicitly. Let $k_1, k_2, k_3 \in \ZZ_2$. Then the image of
\[
\gamma = (a_1 a_2 a_3)^{k_1} (a_2 a_3)^{k_2} a_3^{k_3}
\]
is $(k_1, k_2, k_3) \in \ZZ_2^3$. Since the target group is abelian, this is equal to the image of
\[
a_1^{k_1} a_2^{k_1 + k_2} a_3^{k_1 + k_2 + k_3}.
\]
\end{example}

\begin{corollary}\label{cor:com}
For any $i \in \{1,\ldots,r\}$, the commutator subgroup $[G,G] = N_i \cap \cdots \cap N_{i+r-2}$, and
\[
G_{\text{ab}} \simeq G/N_i \times \cdots \times G/N_{i+r-2} \simeq \ZZ_2^{r-1}.
\]
\end{corollary}

\begin{proof}
By Proposition~\ref{infquo}, the group $G/N_i$ is abelian and topologically generated by the image of $a_i$ for all $i \in \{1,\ldots,r\}$. By Theorem~\ref{Gabsurj}, we have a surjective homomorphism
\[
\phi\colon G_{\text{ab}} \to \prod_{j=1}^{r-1} \ZZ_2.
\]
Moreover, $G_{\text{ab}}$ is generated by the images of any $r-1$ elements from the set $\{a_1, \ldots, a_r\}$. Therefore, $\phi$ is a surjective homomorphism from $\ZZ_2^{r-1}$ to itself. By Lemma~\ref{surjZ2}, $\phi$ is an isomorphism. The kernel of the map $G \to G_{\text{ab}}$ is thus $N_i \cap \cdots \cap N_{i + r - 2}$ for any $i \in \{1, \ldots, r\}$.
\end{proof}

\begin{corollary}\label{cor:com}
If $\gamma = (\gamma_0, \gamma_1)$ is an element of the subgroup $[G,G]$, then $\gamma_0 \gamma_1 \in [G,G]$.
\end{corollary}
\begin{proof}
By Corollary~\ref{cor:com}, $[G,G] = N_1 \cap N_2 \cap \cdots \cap N_{r-1}$, so $\gamma \in N_i$ for all $i = 1, \ldots, r - 1$. By Lemma~\ref{Ni}(3), for each such $i$, the product $\gamma_0 \gamma_1 \in N_{i - 1}$. Considering the indices modulo $r$, we conclude that $\gamma_0 \gamma_1 \in N_r \cap N_1 \cap \cdots \cap N_{r-2} = [G,G]$, as required.
\end{proof}

\begin{corollary}\label{kerN/G}
Let $s = 2$. Then the kernel of the homomorphism $\varphi\colon N \to \Aut(G_{\text{ab}})$ is exactly $G$.
\end{corollary}

\begin{proof}
The result follows immediately from Corollary~\ref{cor:com} and Proposition~\ref{prop:inj}.
\end{proof}

\begin{corollary}
Let $k$ be a number field. Let $f(x) \in k(x)$ be a quadratic rational function whose post-critical set is $\{p_1, p_2, \ldots, p_r\}$ where $f(p_{k-1}) = p_k$ and the critical points are $p_r$ and $p_1$. Then the field of constants associated with $f$ is contained in the cyclotomic field $k\big(\zeta_{2^n} \colon n \geq 1\big).
$
\end{corollary}

%\bibliographystyle{plain}
%\bibliography{refsettled}

\end{document}